%% file: species.tex
\documentclass[letterpaper,11pt]{article}

\usepackage{tikz}
\usepackage{booktabs}
\usepackage{multirow}
\newcommand{\Cplug}{\hat{C}_{\rm seen}}
\newcommand{\ones}{\mathbf{1}}
\newcommand{\DistinctElements}{\textsf{Distinct Elements}\xspace}
\newcommand{\SupportSize}{\textsf{Support Size}\xspace}

\newcommand{\poly}{\mathsf{poly}}

\usepackage{footnote}
\makesavenoteenv{tabular}
\makesavenoteenv{table}

\usepackage[
            CJKbookmarks=true,
            bookmarksnumbered=true,
            bookmarksopen=true,
            colorlinks=true,
            citecolor=red,
            linkcolor=blue,
            anchorcolor=red,
            urlcolor=blue
            ]{hyperref}

\usepackage{ece-paper}
\usepackage{soul}

\newcommand{\Th}{{\rm th}}

\title{Sample complexity of the distinct elements problem}
\author{Yihong Wu\thanks{Department of Statistics and Data Science, Yale University, New Haven, CT, USA, email:~\texttt{yihong.wu@yale.edu}.}
~~~~~~~ Pengkun Yang\thanks{Corresponding author, Department of Electrical and Computer Engineering and the Coordinated Science Lab, University of Illinois at Urbana-Champaign, Urbana, IL, USA, email:~\texttt{pyang14@illinois.edu}.}}

\date{\today}

\begin{document}

\maketitle
\begin{abstract}
    \input{abstract-species}
\end{abstract}



\newpage

\tableofcontents

\input{main-species}

\section*{Acknowledgment}
This research has been supported in part by the National Science Foundation under the grant agreement IIS-14-47879 and CCF-15-27105 and 
an NSF CAREER award CCF-1651588.
The authors thank Greg Valiant for helpful discussion on Lemma 5.14 in his thesis \cite{GV-thesis}.
We thank the anonymous referees for constructive comments which have helped to improve the presentation of the paper.

\input{species.bbl}

\end{document}

%% file: abstract-species.tex
We consider the distinct elements problem, where the goal is to estimate the number of distinct colors in an urn containing $ k $ balls based on $n$ samples drawn with replacements.
Based on discrete polynomial approximation and interpolation, we propose an estimator with additive error guarantee that achieves the optimal sample complexity within $O(\log\log k)$ factors, and in fact within constant factors for most cases. 
The estimator can be computed in $O(n)$ time for an accurate estimation.
The result also applies to sampling without replacement provided the sample size is a vanishing fraction of the urn size.

One of the key auxiliary results is a sharp bound on the minimum singular values of a real rectangular Vandermonde matrix, which might be of independent interest.


%% file: main-species.tex
\graphicspath{{figures/}}
\section{The \DistinctElements problem}
\label{sec:intro}
The \DistinctElements problem \cite{CCMN00} refers to the following question:
\begin{quote}
\emph{Given $ n $ balls randomly drawn from an urn containing $ k $ colored balls, how to estimate the total number of distinct colors in the urn?}
\end{quote}
Originating from ecology, numismatics, and linguistics, this problem is also known as the \emph{species problem} in the statistics literature \cite{Lo92,BF93}.
Apart from the theoretical interests, it has a wide array of applications in various fields, such as estimating the number of species in a population of animals \cite{FCW43,Good1953}, the number of dies used to mint an ancient coinage \cite{Esty86}, and the vocabulary size of an author \cite{ET76}.
In computer science, this problem frequently arises in large-scale databases, network monitoring, and data mining \cite{RRSS09,BJKST02,CCMN00}, where the objective is to estimate the types of database entries or IP addresses from limited observations, since it is typically impossible to have full access to the entire database or keep track of all the network traffic.
The key challenge in the \DistinctElements problem is the following: given a small set of samples where most of the colors are not observed, how to accurately extrapolate the number of unseens?




\subsection{Main results}
The fundamental limit of the \DistinctElements problem is characterized by the sample complexity, i.e., the smallest sample size needed to estimate the number of distinct colors with a prescribed accuracy and confidence level. A formal definition is the following:
\begin{definition}
    \label{def:sample}
    The sample complexity $ n^*(k,\Delta) $ is the minimal sample size $ n $ such that there exists an integer-valued estimator $ \hat C $ based on $n$ balls
                drawn independently with replacements from the urn, such that $\Prob[|\hat C - C| \ge \Delta ] \leq 0.1$ for any urn containing $ k $ balls with $ C $ different colors.\footnote{Clearly, since $ \hat C - C \in \integers $, we shall assume without loss of generality that $ \Delta\in\naturals $, with $ \Delta=1 $ corresponding to the exact estimation of the number of distinct elements.}
\end{definition}

The main results of this paper provide bounds and constant-factor approximations of the sample complexity in various regimes summarized in \prettyref{tab:main}, as well as computationally efficient algorithms.
Below we highlight a few important conclusions drawn from \prettyref{tab:main}:

\begin{description}
    \item[From linear to sublinear: ] From the result for $ k^{0.5+\delta}\le \Delta \le ck $ in \prettyref{tab:main}, we conclude that the sample complexity is sublinear in $k$ if and only if $ \Delta=k^{1-o(1)} $, which also holds for sampling without replacement.
    To estimate within a constant fraction of balls $ \Delta=ck $ for any small constant $ c $, the sample complexity is $ \Theta(\frac{k}{\log k}) $, which coincides with the general support size estimation problem \cite{VV11,WY15} (see \prettyref{sec:related} for a detailed comparison). However, in other regimes we can achieve better performance by exploiting the discrete nature of the \DistinctElements problem.
    
    \item[From linear to superlinear: ] 
                The transition from linear to superlinear sample complexity occurs near $ \Delta=\sqrt{k} $. 
    Although the exact sample complexity near $ \Delta=\sqrt{k} $ is not completely resolved in the current paper, the lower bound and upper bound in \prettyref{tab:main} differ by a factor of at most $ \log\log k $. In particular, the estimator via interpolation can achieve $ \Delta=\sqrt{k} $ with $ n=O(k\log\log k) $ samples, and achieving a precision of $ \Delta\le k^{0.5-o(1)} $ requires strictly superlinear sample size.
\end{description}

\begin{table}[h]
    \centering
    \begin{tabular}{|c|c|c|c|}  
      \toprule
      $ \Delta $                               & Lower bound                        & Upper bound                             & Estimator \\
      \midrule
      $ \le 1 $                               & \multicolumn{2}{c|}{$ \Theta(k\log k) $}                                     & Na\"ive\\
      \hline
                                               & \multicolumn{2}{c|}{\multirow{3}{*}{$ \Theta\pth{k\log \frac{k}{\Delta^2}} $}}   & \\
      $ \qth{1,\sqrt{k}(\log k)^{-\delta}} $         & \multicolumn{2}{c|}{}                                                        & Interpolation \\
                                               & \multicolumn{2}{c|}{}                                                        & (\prettyref{sec:lag})    \\
      \cline{1-3}
      $ \qth{\sqrt{k}(\log k)^{-\delta},k^{0.5+\delta}} $ & $ \Omega\pth{k \pth{1\vee \log \frac{k}{\Delta^2}}} $ & $ O\pth{k\log\frac{\log k}{1 \vee \log\frac{\Delta^2}{k}}} $ &  \\
      \hline
                                                & \multicolumn{2}{c|}{\multirow{3}{*}{$ \Theta\pth{\frac{k}{\log k}\log\frac{k}{\Delta}} $}} &  \\
      $[k^{^{0.5+\delta}},ck]$                     & \multicolumn{2}{c|}{}                                                        & $\ell_2$-approximation\\
                                                & \multicolumn{2}{c|}{}                                                        & (\prettyref{sec:lse})\\
      \cline{1-3}
      $ [ck,(0.5-\delta)k] $                    & $k \exp(-\sqrt{O(\log k \log \log k)})  $\cite{RRSS09}\footnote{A more precise result from \cite{RRSS09} is the following: for $ \Delta\in [ck,0.5k - 2k^{3/4}\sqrt{\log k}] $, $ n^*(k,\Delta)\ge k \exp(-\sqrt{O(\log k (\log \log k+\log\frac{k}{k/2-\Delta}))}) $.}  & $O\pth{\frac{k}{\log k}}$  &    \\
      \bottomrule
    \end{tabular}
    \caption[Summary]{Summary of the sample complexity $ n^*(k,\Delta) $, where $ \delta $ is any sufficiently small constant, $c$ is an absolute positive constant less than 0.5 (same over the table), and the notations $a\wedge b$ and $a\vee b$ stand for $\min\{a,b\}$ and $\max\{a,b\}$, respectively. 
        The estimators are linear with coefficients obtained from either interpolation or $\ell_2$-approximation.
        \label{tab:main}}
\end{table}

To establish the sample complexity, our lower bounds are obtained under zero-one loss and our upper bounds are under the (stronger) quadratic loss.
Hence we also obtain the following characterization of the minimax mean squared error (MSE) of the \DistinctElements problem:
\begin{align*}
  \min_{\hat C}\max_{k\text{-ball urn}}\Expect\pth{\frac{\hat C-C}{k}}^2
  & =\exp\sth{-\Theta\pth{\pth{1\vee \frac{n\log k}{k}} \wedge \pth{\log k \vee \frac{n}{k}} }}\\
  & =\begin{cases}
      \Theta(1) ,& n\le \frac{k}{\log k},\\
      \exp(-\Theta(\frac{n\log k}{k})) ,& \frac{k}{\log k}\le n\le k,\\
      \exp(-\Theta(\log k)) ,& k\le n \le k\log k,\\
      \exp(-\Theta(\frac{n}{k})) , &n \ge k\log k,
  \end{cases}
\end{align*}
where $ \hat C $ denotes an estimator using $ n $ samples with replacements and $ C $ is the number of distinct colors in a $k$-ball urn.

\subsection{Related work}
\label{sec:related}
\paragraph{Statistics literature}

The \DistinctElements problem is equivalent to estimating the number of species (or classes) in a finite population, which has been extensively studied in the statistics (see surveys \cite{BF93,GS04}) and the numismatics literature (see survey \cite{Esty86}).
Motivated by various practical applications, a number of statistical models have been introduced for this problem, the most popular four being (cf. \cite[Figure 1]{BF93}):
\begin{itemize}
    \item \emph{The multinomial model}: $ n $ samples are drawn uniformly at random with replacement;
    \item \emph{The hypergeometric model}: $ n $ samples are drawn uniformly at random without replacement;
    \item \emph{The Bernoulli model}: each individual is observed independently with some fixed probability, and thus the total number of samples is a binomial random variable;
    \item \emph{The Poisson model}: the number of observed samples in each class is independent and Poisson distributed, and thus the total sample size is also a Poisson random variable.
\end{itemize}
These models are closely related: conditioned on the sample size, the Bernoulli model coincides with the hypergeometric one, and Poisson model coincides with the multinomial one; furthermore, hypergeometric model can simulate multinomial one and is hence more informative.
The multinomial model is adopted as the main focus of this paper and the sample complexity in \prettyref{def:sample} refers to the number of samples with replacement.
In the undersampling regime where the sample size is significantly smaller than the population size, all four models are approximately equivalent.
See \prettyref{app:model} for a rigorous justification and detailed comparisons.

Under these models various estimators have been proposed  such as unbiased estimators \cite{goodman1949}, Bayesian estimators \cite{Hill79}, variants of Good-Turing estimators \cite{Chao92}, etc. None of these methodologies, however, have a provable worst-case guarantee. 
Finally, we mention a closely related problem of estimating the number of connected components in a graph based on sampled induced subgraphs. 
In the special case where the underlying graph consists of disjoint cliques, the problem is exactly equivalent to the \DistinctElements problem \cite{Frank78}.

\paragraph{Computer science literature}
The interests in the \DistinctElements problem also arise in the database literature, where various intuitive estimators \cite{HOT88,NS90} have been proposed under simplifying assumptions such as uniformity, and few performance guarantees are available.
More recent work in \cite{CCMN00,BKS01} obtained the optimal sample complexity under the \emph{multiplicative} error criterion, where the minimum sample size to estimate the number of distinct elements within a factor of $\alpha $ is shown to be $ \Theta(k/\alpha^2) $.
For this task, it turns out the least favorable scenario is to distinguish an urn with unitary color from one with \emph{almost} unitary color, the impossibility of which implies large multiplicative error.
However, the optimal estimator performs poorly compared with others on an urn with many distinct colors \cite{CCMN00}, the case where most estimators enjoy small multiplicative error.
In view of the limitation of multiplicative error, additive error is later considered by \cite{RRSS09,Valiant11}.
To achieve an additive error of $ ck $ for a constant $ c\in(0,\frac{1}{2}) $, the result in \cite{CCMN00} only implies an $ \Omega(1/c) $ sample complexity lower bound, whereas a much stronger lower bound scales like $ k^{1 - O(\sqrt{\frac{\log\log k}{\log k}})} $ obtained in \cite{RRSS09}, which is almost linear.
Determining the optimal sample complexity under additive error is the focus of the present paper.

The \DistinctElements problem can be viewed as a special case of the \SupportSize problem, where the goal is to estimate the cardinality of the support of an unknown discrete distribution, whose nonzero probabilities are at least $ \frac{1}{k} $, based on independent samples. Improving previous results in \cite{VV11}, the optimal sample complexity has been recently determined in \cite{WY15} to be 
\begin{equation}
\Theta\pth{\frac{k}{\log k}\log^2\frac{k}{\Delta}}.
\label{eq:sample-supp}
\end{equation}
 Samples drawn from a $ k $-ball urn with replacement can be viewed as \iid samples from a distribution supported on the set $ \{\frac{1}{k},\frac{2}{k},\dots,\frac{k}{k}\} $.
From this perspective, any support size estimator, as well as its performance guarantee, is applicable to the \DistinctElements problem.

We briefly describe and compare the strategy to construct estimators in \cite{WY15} and the current paper.
Both are based on the idea of \emph{polynomial approximation}, a powerful tool to circumvent the nonexistence of unbiased estimators \cite{LNS99}.
The key is to approximate the function to be estimated by a polynomial, whose degree is chosen to balance the approximation error (bias) and the estimation error (variance).
The worst-case performance guarantee for the \SupportSize problem 
in \cite{WY15} is governed by the uniform approximation error over an interval where the probabilities may reside.
In contrast, in the \DistinctElements problem, samples are generated from a distribution supported on a \emph{discrete} set of values.
Uniform approximation over a discrete subset leads to smaller approximation error and, in turn, improved sample complexity.
It turns out that $ O(\frac{k}{\log k}\log\frac{k}{\Delta}) $ samples are sufficient to achieve an additive error of $ \Delta $ that satisfies $ k^{0.5+O(1)} \le \Delta \le O(k) $, which strictly improves the 
sample complexity \prettyref{eq:sample-supp} for the \SupportSize problem, thanks to the discrete structure of the \DistinctElements problem.


The \DistinctElements problem considered here is not to be confused with the formulation in the streaming literature, where the goal is to approximate the number of distinct elements in the observations with low space complexity, see, e.g., \cite{FFGM07,KNW10}.
There, the proposed algorithms aim to optimize the memory consumption, but still require a full pass of every ball in the urn.
This is different from the setting in the current paper, where only random samples drawn from the urn are available.

\subsection{Organization}
The paper is organized as follows:
In \prettyref{sec:ub} we describe a unified approach to construct estimators via discrete polynomial approximation, whose bias is analyzed in \prettyref{sec:lse} and variance is upper bounded in Sections \ref{sec:sigma-min} and \ref{sec:lag} separately.
In \prettyref{sec:lb} we obtain lower bounds on the sample complexity in \prettyref{tab:main} which establish the optimality of the proposed estimators.
\prettyref{sec:table} explains how sample complexity bounds summarized  in \prettyref{tab:main} follow from various results in Sections \ref{sec:ub} and \ref{sec:lb}.
Connections between the four sampling model mentioned in \prettyref{sec:related} are detailed in \prettyref{app:model}. 
Proofs of auxiliary results are deferred to \prettyref{app:corr} and  
\prettyref{app:aux}.

\subsection{Notations}
All logarithms are with respect to the natural base. 
The transpose of a matrix $ A $ is denoted by $ A^\top $.
Let $ \ones $ denote the all-one column vector.
Let $\|\cdot\|_p$ denote the vector $\ell_p$-norm, for $1 \leq p\leq \infty$.
Let $ \Poi(\lambda) $ be the Poisson distribution with mean $ \lambda $, $ \Bern(p) $ be the Bernoulli distribution with mean $ p $, $ \Binom(n,p) $ be the binomial distribution with $ n $ trials and success probability $ p $, and $ \Hypergeo(N,K,n) $ be 
the hypergeometric distribution with probability mass function $\binom{K}{k}\binom{N-K}{n-k}/{\binom{N}{n}}$, 
for $0\vee (n+K-N) \leq k \leq n\wedge K$.
The $ n $-fold product of a distribution $ P $ is denoted by $ P^{\otimes n} $.
We use standard big-$ O $ notations: for any positive sequence $ \{a_n\} $ and $ \{b_n\} $, $ a_n=O(b_n) $ or $ a_n\lesssim b_n $ if $ a_n\le cb_n $ for some absolute constant $ c>0 $, or equivalently, $\sup_n \frac{a_n}{b_n} < \infty$; $ a_n=\Omega(b_n) $ or $ a_n\gtrsim b_n $ if $ b_n= O(a_n) $; $ a_n=\Theta(b_n) $ or $ a_n\asymp b_n $ if both $ a_n=O(b_n) $ and $ b_n=O(a_n) $; $ a_n=o(b_n) $ if $ \lim a_n/b_n=0 $; $ a_n=\omega(b_n) $ if $ b_n=o(a_n) $.
Furthermore, the subscript in $o_n(1)$ indicates convergence in $n$ that is uniform in all other parameters.
We use the notations $a\wedge b$ and $a\vee b$ for $\min\{a,b\}$ and $\max\{a,b\}$, respectively.
For $M\in\naturals$, let $[M]\triangleq \{1,\dots,M\}$.
For $\alpha\in\reals$ and $S\subset \reals$, let $\alpha S\triangleq \{\alpha x:x\in S\}$.

\section{Linear estimators via discrete polynomial approximation}
\label{sec:ub}
In this section we develop a unified framework to construct linear estimators and analyze its performance. Note that linear estimators (i.e.~linear combinations of fingerprints) have been previously used for estimating distribution functionals \cite{Paninski04,VV11,VV11-focs,WY15}.
As commonly done in the literature, we assume the \emph{Poisson sampling model}, where the sample size is a random variable $ \Poi(n) $ instead of being exactly $ n $.
Under this model, the histograms of the samples, which count the number of balls in each color, are independent which simplifies the analysis.
Any estimator under the Poisson sampling model can be easily modified for fixed sample size, and vice versa, thanks to the concentration of the Poisson random variable near its mean.
Consequently, the sample complexities of these two models are close to each other, as shown in \prettyref{cor:nstar-connection} in \prettyref{app:model}.

\subsection{Performance guarantees for general linear estimators}
        Recall that $C$ denotes the number of distinct colors in a urn containing $k$ colored balls.
        Let $k_i$ denote the number of balls of the $ i\Th$ color in the urn. Then $\sum_i k_i = k$ and $C = \sum_i \indc{k_i > 0}$.
        Let $X_1,X_2,\ldots$ be independently drawn with replacement from the urn. Equivalently, the $X_i$'s are \iid according to a distribution $P = (p_i)_{i\geq 1}$, where $p_i=k_i/k$ is the fraction of balls of the $i\Th$ color.
        The observed data are $X_1,\ldots,X_N$, where the sample size $N$ is independent from $(X_i)_{i\geq 1}$ and is distributed as $\Poi(n)$.
        Under the Poisson model (or any of the sampling models described in \prettyref{sec:related}), the \emph{histograms} $\{N_i\}$ are sufficient statistics for inferring any aspect of the urn configuration; here $ N_i $ is the number of balls of the $ i\Th$ color observed in the sample, which is independently distributed as 
        $ \Poi(np_i) $.
        Furthermore, the \emph{fingerprints} $\{\Phi_j\}_{j\geq 1}$, which are the histogram of the histograms,         are also sufficient for estimating any permutation-invariant distributional property \cite{Paninski03,Valiant11}, in particular, the number of colors.
        Specifically, the $j$th fingerprint $ \Phi_j $ denotes the number of colors that appear exactly $ j $ times.
Note that $U \triangleq \Phi_0 $, the number of unseen colors, is not observed.

The na\"ive estimator, ``what you see is what you get,'' is simply the number of observed distinct colors, which can be expressed in terms of fingerprints as
\begin{equation*}
    \Cplug = \sum_{j\ge 1}\Phi_j,
\end{equation*}
This is typically an underestimator because $C = \Cplug + U$. In turn, our estimator is
\begin{equation}
    \tilde{C} = \Cplug + \hat{U},
    \label{eq:tildeC-correction}
\end{equation}
which adds a linear correction term
\begin{equation}
                \hat{U}=\sum_{j\geq 1} u_j\Phi_j,
    \label{eq:hatU}
\end{equation}
where the coefficients $u_j$'s are to be specified.
Since the fingerprints $ \Phi_0,\Phi_1,\dots $ are dependent (for example, they sum up to $C$),  \prettyref{eq:hatU} serves as a linear predictor of $U=\Phi_0$ in terms of the observed fingerprints.
Equivalently, in terms of histograms, the estimator has the following decomposable form:
\begin{equation}
    \tilde{C}=\sum_{i=1}^\infty g(N_i),
    \label{eq:tildeC-sum}
\end{equation}
where $g:\integers_+ \to \reals$ satisfies  $ g(0)=0 $ and $ g(j)=1+u_j $ for $j \geq 1$.
In fact, any estimator that is linear in the fingerprints can be expressed of the decomposable form \prettyref{eq:tildeC-sum}.


The main idea to choose the coefficients $ u_j $ is to achieve a good trade-off between the variance and the bias.
In fact, it is instructive to point out that linear estimators can easily achieve exactly zero bias, which, however, comes at the price of high variance. To see this, note that the bias of the estimator \prettyref{eq:tildeC-sum} is 
$\Expect[\tilde C]-C = \sum_{i\ge 1} (\Expect[g(N_i)]-1)$, where
\begin{equation}
    |\Expect[g(N_i)-1]|
    =e^{-np_i}\abs{-1+\sum_{j= 1}^\infty k_i^j\frac{u_j(n/k)^j}{j!}}
                \leq e^{-n/k}\max_{a\in[k]}\abs{\phi(a)-1},
    \label{eq:lp-full}
\end{equation}
and $\phi(a) \triangleq \sum_{j\ge 1}a^j\frac{u_j(n/k)^j}{j!}$ is a (formal) power series with $\phi(0)=0$.
The right-hand side of \prettyref{eq:lp-full} can be made zero by choosing $\phi$ to be, \eg, the Lagrange interpolating polynomial that satisfies $ \phi(0)=-1 $ and $ \phi(i)=0 $ for $ i\in[k] $, namely, $\phi(a) = \frac{(-1)^{k+1}}{k!}\prod_{i=1}^k (a-i)$; however, this strategy results in a high-degree polynomial $\phi$ with large coefficients, which, in turn, leads to a large variance of the estimator.

To reduce the variance of our estimator, we only use the first $L$ fingerprints in \prettyref{eq:hatU} by setting $u_j = 0$ for all $j > L$, where $L$ is chosen to be proportional to $\log k$.
This restricts the polynomial degree in \prettyref{eq:lp-full} to at most $ L $ and, while possibly incurring bias, reduces the variance.
A further reason for only using the first few fingerprints is that higher-order fingerprints are \emph{almost uncorrelated} with the number of unseens $\Phi_0$. For instance, if red balls are observed for $n/2$ times, the only information this reveals is that approximately half of the urn are red. In fact, the correlation between $\Phi_0$ and $\Phi_j$
decays exponentially (see \prettyref{app:corr} for a proof). 
Therefore for $L=\Theta(\log k)$, $\{\Phi_j\}_{j > L}$ offer little predictive power about $\Phi_0$.
Moreover, if a color is observed at most $L$ times, say, $ N_i\le L $, this implies that, with high probability, 
$ k_i\le M $, where $ M = O(kL/n)$, thanks to the concentration of Poisson random variables.
Therefore, effectively we only need to consider those colors that appear in the urn for at most $M$ times, i.e., $ k_i\in [M] $, 
for which the bias is at most
\begin{equation}
    |\Expect[g(N_i)-1]|
                \le e^{-n/k}\max_{a\in[M]} \abs{\phi(a)-1} = e^{-n/k}\max_{x\in[M]/M} \abs{p(x)-1} 
    = e^{-n/k} \norm{Bw-\ones}_{\infty},    
    \label{eq:lp}
\end{equation}
where $p(x) \triangleq \phi(Mx) = \sum_{j=1}^L w_j x^j$, $ w=(w_1,\dots,w_L)^\top $, and
\begin{equation}
    w_j 
                \triangleq \frac{u_j(Mn/k)^j}{j!},\quad
    B \triangleq 
    \begin{pmatrix}
        1/M & (1/M)^2 & \cdots & (1/M)^L  \\
        2/M & (2/M)^2 & \cdots & (2/M)^L  \\
        \vdots  & \vdots  & \ddots & \vdots  \\
        1 & 1 & \cdots & 1
    \end{pmatrix}
    \label{eq:Bw}
\end{equation}
is a (partial) Vandermonde matrix.
Lastly, since $ \Cplug\le C\le k $, we define the final estimator to be $\tilde C$ projected to the interval $[\Cplug, k]$.
We have the following error bound:
\begin{prop}
    \label{prop:rate-w}
    Assume the Poisson sampling model.
    Let
    \begin{equation}
        L=\alpha \log k, \quad M=\frac{\beta  k\log k}{n},
                                \label{eq:LM}
    \end{equation}
    for any $ \beta > \alpha $ such that $ L $ and $ M $ are integers. Let $ w\in\reals^{L} $.
                Let $\tilde C$ be defined in \prettyref{eq:tildeC-correction} with $u_j = w_j j! (\frac{k}{nM})^j$ for $j\in [L]$ and $u_j=0$ otherwise.
                Define $ \hat C\triangleq(\tilde C \vee \Cplug)\wedge k $.
    Then
    \begin{equation}
        \Expect{(\hat C-C)^2}
        \le k^2 e^{-2n/k}\norm{Bw-\ones}_{\infty}^2+ke^{-n/k}+k\max_{m\in[M]}\Expect_{N\sim\Poi(nm/k)}[u_{N}^2]+k^{-(\beta -\alpha \log\frac{e\beta }{\alpha }-3)}.
        \label{eq:rate-w}
    \end{equation}
\end{prop}
\begin{proof}
    Since $ \Cplug\le C\le k $, $ \hat{C} $ is always an improvement of $ \tilde{C} $.
    Define the event $ E\triangleq \cap_{i=1}^k\{N_i\le L\Rightarrow kp_i\le M\} $, which means that whenever $N_i\le L$ we have $p_i\le M/k$. 
    Since $ \beta>\alpha $, applying the Chernoff bound and the union bound yields $ \Prob[E^c]\le k^{1-\beta +\alpha \log\frac{e\beta }{\alpha }} $, and thus
    \begin{equation}
        \Expect{(\hat C-C)^2}\le \Expect((\hat C-C)\Indc_E)^2+k^2\Prob[E^c]
        \le \Expect((\tilde C-C)\Indc_E)^2+k^{3-\beta +\alpha \log\frac{e\beta }{\alpha }}.
        \label{eq:hatC-risk1}
    \end{equation}
    The decomposable form of $ \tilde{C} $ in \prettyref{eq:tildeC-sum} leads to
    \begin{equation*}
        (\tilde C-C)\Indc_E
        =\sum_{i:k_i\in [M]}(g(N_i)-1)\indc{N_i\le L}\triangleq\calE.
    \end{equation*}
    In view of the bias analysis in \prettyref{eq:lp}, we have
    \begin{equation}
        |\Expect[{\cal E}]|
        \le \sum_{i:k_i\in [M]}e^{-nk_i/k}\norm{Bw-\ones}_{\infty}
        \le ke^{-n/k}\norm{Bw-\ones}_{\infty}.
        \label{eq:calE-bias}
    \end{equation}
    Recall that $g(0)=0$ and $g(j)=u_j+1$ for $j\in[L]$. Since $ N_i$ is independently distributed as $ \Poi(nk_i/k)$, we have
    \begin{align}
      \var[{\cal E}]
      & = \sum_{i:k_i\in [M]}\var\qth{(g(N_i)-1)\indc{N_i\le L}}
        \le \sum_{i:k_i\in [M]}\Expect\qth{(g(N_i)-1)^2\indc{N_i\le L}}\nonumber\\
      & = \sum_{i:k_i\in [M]}\pth{e^{-nk_i/k}+\Expect [u_{N_i}^2]}
        \le ke^{-n/k}+k\max_{m\in[M]}\Expect_{N\sim\Poi(nm/k)}[u_N^2]. \label{eq:calE-var}
    \end{align}

    Combining the upper bound on the bias in \prettyref{eq:calE-bias} and the variance in \prettyref{eq:calE-var} yields an upper bound on $ \Expect[\calE^2] $.
    Then the MSE in \prettyref{eq:rate-w} follows from \prettyref{eq:hatC-risk1}.
\end{proof}

\prettyref{prop:rate-w} suggests that the coefficients of the linear estimator can be chosen by solving the following linear programming (LP):
\begin{equation}
\min_{w \in \reals^L} \norm{Bw-\ones}_{\infty}
\label{eq:LP}
\end{equation}
and showing that the solution 
does not have large entries.
Instead of the $\ell_\infty$-approximation problem \prettyref{eq:LP}, whose optimal value is difficult to analyze, we solve the $\ell_2$-approximation problem as a relaxation:
\begin{equation}
\min_{w \in \reals^L} \Norm{Bw-\ones}_2,
\label{eq:L2}
\end{equation}
which is an upper bound of \prettyref{eq:LP}, and is in fact within an $O(\log k)$ factor since $M=O(k\log k/n)$ and $n=\Omega(k/\log k)$.
In the remainder of this section, we consider two separate cases:
\begin{itemize}
        \item $M>L$ ($n \lesssim k$): In this case, the linear system 
in \prettyref{eq:L2} is overdetermined and the minimum is non-zero. Surprisingly, as shown in \prettyref{sec:lse},
the exact optimal value can be found in closed form using discrete orthogonal polynomials. 
The coefficients of the solution can be bounded using the minimum singular value of the matrix $ B $, which is analyzed in \prettyref{sec:sigma-min} .

        \item $M\leq L$ ($n \gtrsim k$): In this case, the linear system is underdetermined and the minimum in \prettyref{eq:L2} is zero. To bound the variance, it turns out that the coefficients bound obtained from the minimum singular value is not precise enough in this regime. Instead, we express the coefficients in terms of Lagrange interpolating polynomials and use Stirling numbers to obtain sharp variance bounds. This analysis in carried out in \prettyref{sec:lag}.

\end{itemize}
%

We finish this subsection with two remarks:

\begin{remark}[Discrete versus continuous approximation]
The optimal estimator for the \SupportSize problem in \cite{WY15} has the same linear form as \prettyref{eq:tildeC-correction}; however, since the probabilities can take any values in an interval, the coefficients are found to be the solution of the continuous polynomial approximation problem
\begin{equation}
\inf_{p} \max_{x \in [\frac{1}{M},1]}  |p(x) - 1| = \exp\Big(-\Theta\Big(\frac{L}{\sqrt{M}}\Big)\Big).
\label{eq:cheby}
\end{equation}
        where the infimum is taken over all degree-$L$ polynomials such that $p(0)=0$, achieved by the (appropriately shifted and scaled) Chebyshev polynomial \cite{timan63}.
        In contrast, in \prettyref{sec:lse} we show that the discrete version of \prettyref{eq:cheby}, which is equivalent to the LP \prettyref{eq:LP}, satisfies
        \begin{equation}
\inf_{p} \max_{x \in \{\frac{1}{M},\frac{2}{M},\ldots,1\}}  |p(x) - 1| = \poly(M) \exp\Big(-\Theta\Big(\frac{L^2}{M}\Big)\Big),
\label{eq:cheby-discrete}
\end{equation}
provided $L < M$.
The difference between \prettyref{eq:cheby} and \prettyref{eq:cheby-discrete} explains why the sample complexity \prettyref{eq:sample-supp} for the \SupportSize problem has an extra log factor compared to that of the \DistinctElements problem in \prettyref{tab:main}.
When the sample size $n$ is large enough, interpolation is used in lieu of approximation. 
See \prettyref{fig:approx} for an illustration.

\begin{figure}[ht]
    \centering
    \subfigure[Continuous approximation]
    {\label{fig:cont-approx} 
        \includegraphics[width=.31\linewidth]{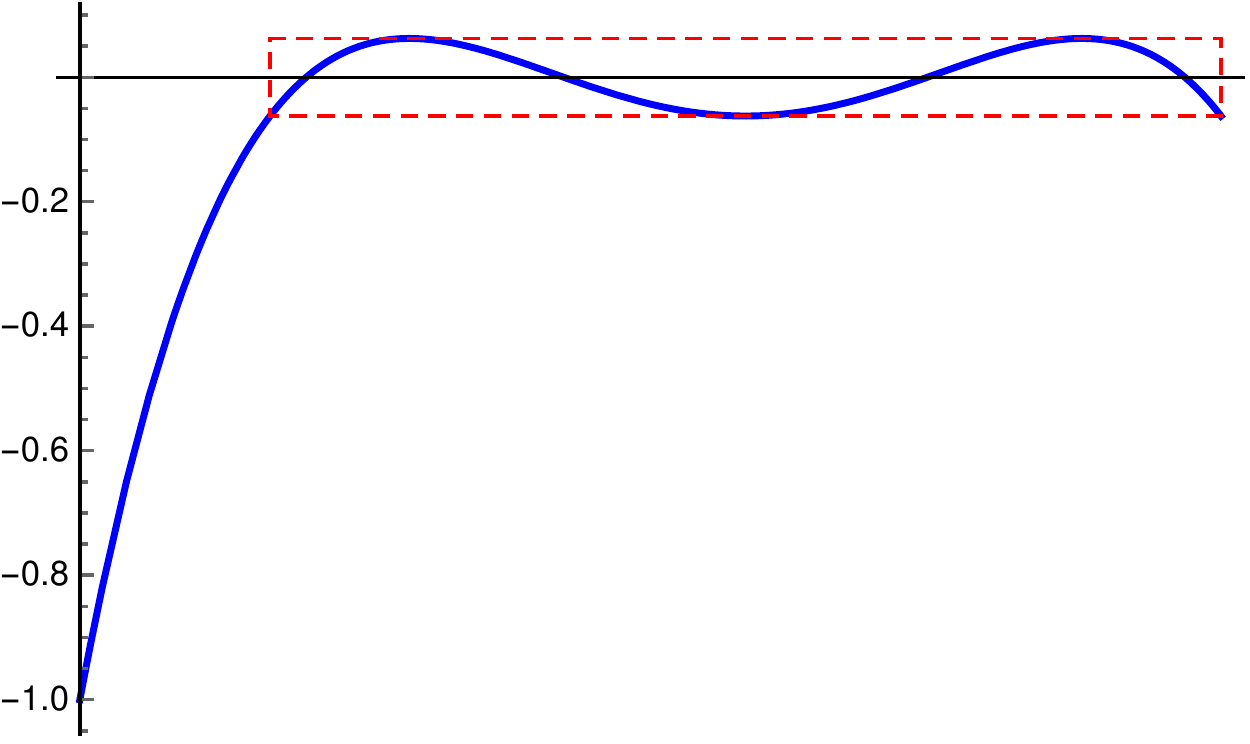}}
    \subfigure[Discrete approximation]
    {\label{fig:discrete-approx} 
        \includegraphics[width=.31\linewidth]{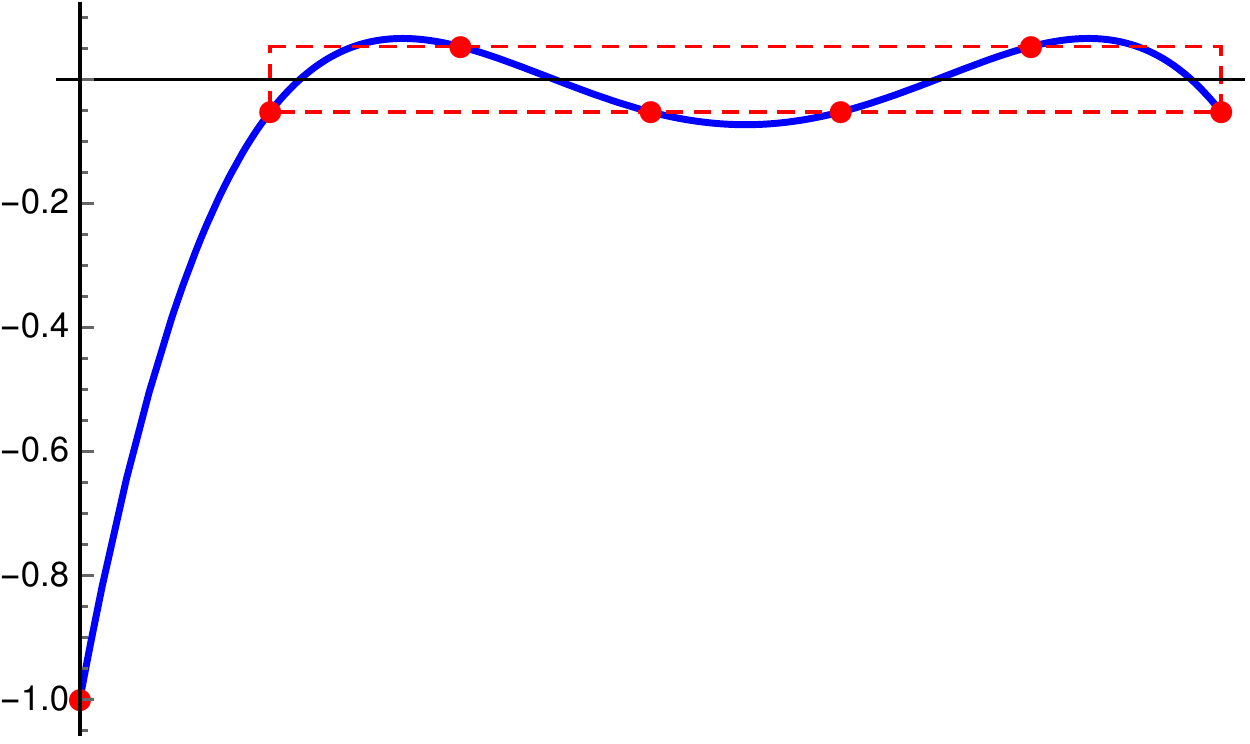}}
    \subfigure[Interpolation]
    {\label{fig:interpolation} 
        \includegraphics[width=.31\linewidth]{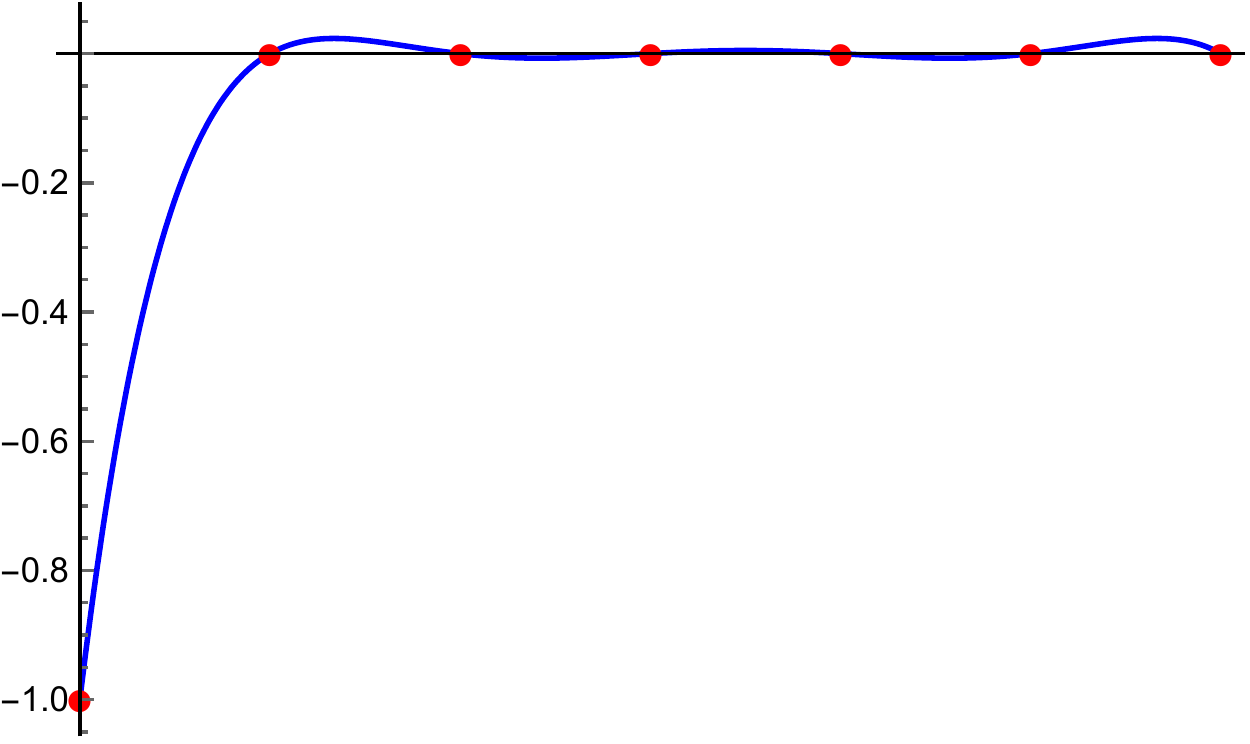}}
    \caption{Continuous and discrete polynomial approximations for $M= 6 $ and degree $L=4$, where \subref{fig:cont-approx} and \subref{fig:discrete-approx} plot the optimal solution to \prettyref{eq:cheby} and \prettyref{eq:cheby-discrete} respectively. The interpolating polynomial in \subref{fig:interpolation} requires a higher degree $L=6$. }
    \label{fig:approx}
\end{figure}

\end{remark}

\begin{remark}[Time complexity] 
    The time complexity of the estimator \prettyref{eq:tildeC-correction} consists of:
    (a) Computing histograms $ N_i $ and fingerprints $ \Phi_j $ of $ n $ samples: $ O(n) $;
    (b) Computing the coefficients $ w $ by solving the least square problem in \prettyref{eq:lp}: $ O(L^2(M+L)) $;
    (c) Evaluating the linear combination \prettyref{eq:tildeC-correction}: $ O(n\wedge k) $.
    As shown in \prettyref{tab:main}, for an accurate estimation the sample complexity is $ n=\Omega(\frac{k}{\log k}) $, which implies $ L=O(\log k) $ and $ M=O(\log^2k) $. Therefore, the overall time complexity is $ O(n+\log^4k)=O(n)$.
\end{remark}

\subsection{Exact solution to the $\ell_2$-approximation}
\label{sec:lse}

Next we give an explicit solution to the $\ell_2$-approximation problem \prettyref{eq:L2}.
In general, the optimal solution is given by $ w^*=(B^\top B)^{-1}B^\top\ones $ and the minimum value is the Euclidean distance between the all-one vector $\ones$ and the column span of $B$, which, in the case of $ M > L $, is non-zero (since $B$ has linearly independent columns). 
Taking advantage of the Vandermonde structure of the matrix $B$ in \prettyref{eq:Bw}, we note that \prettyref{eq:L2} can be interpreted as finding the orthogonal projection of the constant function onto the linear space of polynomials of degree between $1$ and $L$ defined on the discrete set $[M]/M$. Using the 
orthogonal polynomials with respect to the counting measure, known as \emph{discrete Chebyshev (or Gram) polynomials} (see \cite[Section 2.8]{orthogonal.poly} or \cite[Section 2.4.2]{discrete-orthogonal}), we show that, surprisingly, the optimal value of the $\ell_2$-approximation can be found in closed form:


\begin{lemma}
    \label{lmm:lse-min}
    For all $L\geq 1$ and $ M\ge L+1 $,
    \begin{equation}
        \min_{w\in\reals^L}\norm{Bw-\ones}_2
        =\qth{{\frac{\binom{M+L+1}{L+1}}{\binom{M}{L+1}}-1}}^{-1/2} = \qth{\exp\pth{\Theta\pth{\frac{L^2}{M}}}-1}^{-1/2}.
        \label{eq:lse-min}
    \end{equation}
\end{lemma}


\begin{proof}
    Define the following inner product between functions $f$ and $g$:
    \begin{equation}
        \inner{f,g}\triangleq \sum_{i=1}^{M}f\pth{\frac{i}{M}}g\pth{\frac{i}{M}}
        \label{eq:inner}
    \end{equation}
    and the induced norm $ \norm{f}\triangleq \sqrt{\inner{f,f}} $.
    The least square problem \prettyref{eq:lse-min} can be equivalently formulated as
    \begin{equation}
        \min_{w\in\reals^L}\Norm{-1+w_1x+w_2x^2+\dots+w_Lx^L}.
        \label{eq:lse-poly}
    \end{equation}
    This can be analyzed using the orthogonal polynomials under the inner product \prettyref{eq:inner}, which we describe next.

                Recall the discrete Chebyshev polynomial \cite[Sec.~2.8]{orthogonal.poly}: for $x=0,1,\dots,M-1$,
                \begin{equation}
        t_m(x)
        \triangleq \frac{1}{m!}\Delta^mp_m(x)
        =\frac{1}{m!} \sum_{j=0}^{m}(-1)^j\binom{m}{j}p_m(x+m-j), \quad 0\le m\le M-1,
        \label{eq:tn-full}
    \end{equation} 
    where 
    \begin{equation}
        p_m(x)\triangleq x(x-1)\cdots(x-m+1)(x-M)(x-M-1)\cdots(x-M-m+1),
        \label{eq:def-pm}
    \end{equation}
    and $ \Delta^m $ denotes the $ m $-th order forward difference.
                The polynomials $\{t_0,\ldots,t_{M-1}\}$ are orthogonal with respect to the counting measure over the discrete set $ \sth{0,1,\dots,M-1} $; in particular, we have
    (cf.~\cite[Sec.~2.8.2, 2.8.3]{orthogonal.poly}):
    \begin{align*}
        \sum_{x=0}^{M-1}t_m(x)t_\ell(x)&=0,\quad m\ne \ell,\\
        \sum_{x=0}^{M-1}t_m^2(x)&=\frac{M(M^2-1^2)(M^2-2^2)\cdots(M^2-m^2)}{2m+1}\triangleq c(M,m).
    \end{align*}

    By appropriately shifting and scaling the set of polynomials $ t_m $,
    we define an orthonormal basis for the set of polynomials of degree at most $ L \leq M-1$ under the inner product \prettyref{eq:inner} by 
    \begin{equation}
        \phi_m(x)=\frac{t_m(M x-1)}{\sqrt{c(M,m)}},\quad m=0,\dots,L.
        \label{eq:def-phi}
    \end{equation}
    Since $\{\phi_m\}_{m=0}^L$ constitute a basis for polynomials of degree at most $L$, the least square problem \prettyref{eq:lse-poly} can be equivalently formulated as
    \begin{equation*}
        \min_{a:\sum_{i=1}^L a_i\phi_i(0)=-1}\norm{\sum_{i=0}^{L}a_i\phi_i}
        =\min_{a:\iprod{a}{\phi(0)}=-1}\norm{a}_2,
    \end{equation*}
    where $ \phi(0)\triangleq (\phi_0(0),\dots,\phi_L(0)) $, $a=(a_0,\ldots,a_L)$, and $\iprod{\cdot}{\cdot}$ denotes vector inner product.
    Thus, the optimal value is clearly $ \frac{1}{\norm{\phi(0)}_2} $, achieved by $ a^*=-\frac{\phi(0)}{\norm{\phi(0)}_2^2} $.

     From \prettyref{eq:def-pm} we have $ p_m(0)=p_m(1)=\dots=p_m(m-1)=0 $.
     By the formula of $ t_m $ in \prettyref{eq:tn-full}, we obtain
     \begin{equation*}
         t_m(-1)=\frac{1}{m!}(-1)^mp_m(-1)=(-1)^m\prod_{j=1}^m(M+j).
     \end{equation*}
    In view of the definition of $ \phi_m $ in \prettyref{eq:def-phi}, we have
    \begin{equation*}
        \phi_m(0)
        =\frac{t_m(-1)}{\sqrt{c(M,m)}}
        =\frac{(-1)^m\prod_{j=1}^m(M+j)}{\sqrt{\frac{M\prod_{j=1}^m(M^2-j^2)}{2m+1}}}
        =(-1)^m\sqrt{\frac{2m+1}{M}\prod_{j=1}^m\frac{M+j}{M-j}}. 
    \end{equation*}
    Therefore
    \begin{equation*}
        \norm{\phi(0)}_2^2
        =\sum_{m=0}^L\frac{2m+1}{M}\prod_{j=1}^m\frac{M+j}{M-j}
        =\frac{\binom{M+L+1}{L+1}}{\binom{M}{L+1}} - 1,
    \end{equation*}
    where the last equality follows from induction since
    \begin{equation*}
        \frac{\binom{M+L+1}{L+1}}{\binom{M}{L+1}} -\frac{\binom{M+L}{L}}{\binom{M}{L}} 
        =\frac{2L+1}{M}\prod_{j=1}^L\frac{M+j}{M-j}.
    \end{equation*}
                This proves the first equality in \prettyref{eq:lse-min}.


    The second equality in \prettyref{eq:lse-min} is a direct consequence of Stirling's approximation.
    If $ M=L+1 $, then
    \begin{equation}
        \frac{\binom{M+L+1}{L+1}}{\binom{M}{L+1}}
        = \binom{2(L+1)}{L+1}
        = \exp(\Theta(L)).
        \label{eq:lse-ub-ref1}
    \end{equation}
    If $ M\ge L+2 $, denoting $ x=\frac{L+1}{M} $ and applying $ n!= \sqrt{2\pi n}(\frac{n}{e})^n(1+\Theta(\frac{1}{n})) $ when $ n\ge 1 $, we have
    \begin{align}
      \frac{\binom{M+L+1}{L+1}}{\binom{M}{L+1}}
      & = \frac{(M+L+1)!(M-L-1)!}{(M!)^2}
        = \frac{(M(1+x))!(M(1-x))!}{(M!)^2}\nonumber\\
      & = \frac{\sqrt{2\pi M(1+x)}(\frac{M(1+x)}{e})^{M(1+x)} \sqrt{2\pi M(1-x)}(\frac{M(1-x)}{e})^{M(1-x)}(1+\Theta(\frac{1}{M(1+x)}+\frac{1}{M(1-x)}))}{2\pi M(\frac{M}{e})^{2M}(1+\Theta(\frac{1}{M}))}\nonumber\\
      & = \sqrt{1-x^2}\exp\pth{M( (1+x)\log(1+x)+(1-x)\log(1-x) )}\frac{1+\Theta(\frac{1}{M(1-x^2)})}{1+\Theta(\frac{1}{M})} \nonumber\\
      & = \exp\pth{\Theta(Mx^2)+\frac{1}{2}\log(1-x^2)+\log\frac{1+\Theta(\frac{1}{M(1-x^2)})}{1+\Theta(\frac{1}{M})}},\label{eq:lse-ub-ref2}
    \end{align}
    where the last step follows from $ (1+x)\log(1+x)+(1-x)\log(1-x)= \Theta(x^2) $ when $ 0\le x\le 1 $.
    In the exponent of \prettyref{eq:lse-ub-ref2}, the term $ \Theta(Mx^2) $ dominates when $ M\ge L+2 $.
    Applying \prettyref{eq:lse-ub-ref1} and \prettyref{eq:lse-ub-ref2} to the exact solution \prettyref{eq:lse-min} yields the desired approximation.
\end{proof}

\subsection{Minimum singular values of real rectangle Vandermonde matrices}
\label{sec:sigma-min}
In \prettyref{prop:rate-w} the variance of our estimator is bounded by the magnitude of coefficients $ u $, which is related to the polynomial coefficients $ w $ by \prettyref{eq:Bw}.
A classical result from approximation theory is that if a polynomial is bounded over a compact interval, its coefficients are at most exponential in the degree \cite[Theorem 2.9.11]{timan63}:
for any degree-$L$ polynomial $p(x) = \sum_{i=0}^{L}w_ix^i$,
\begin{equation}
    \max_{0\le i\le L}|w_i|
    \le \max_{x\in[0,1]}\abs{p(x)}\exp(O(L)),
    \label{eq:timan-coeffs}
\end{equation}
which is tight when $p$ is the Chebyshev polynomial.
This fact has been applied in statistical contexts to control the variance of estimators obtained from best polynomial approximation \cite{CL11,WY14,WY15,JVHW15}. In contrast, for the \DistinctElements problem, the polynomial is only known to be bounded over the discretized interval.
Nevertheless, we show that the bound  \prettyref{eq:timan-coeffs} continues to hold as long as the discretization level exceeds the degree:
\begin{equation}
    \max_{0\le i\le L}|w_i|\le \max_{x\in\{\frac{1}{M},\frac{2}{M},\dots,1\}}\abs{p(x)}\exp(O(L)),
                \label{eq:timan-coeffs-discrete}
\end{equation}
provided that $ M\ge L+1 $ (see \prettyref{rmk:timan-coeffs-discrete} after \prettyref{lmm:sigma-min}).
Clearly, \prettyref{eq:timan-coeffs-discrete} implies \prettyref{eq:timan-coeffs} by sending $M\diverge$. 
If $M \leq L$, a coefficient bound like \prettyref{eq:timan-coeffs-discrete} is impossible, because one can add to $p$ an arbitrary degree-$L$ interpolating polynomial that evaluates to zero at all $M$ points.

To bound the coefficients, note that the optimal solution of $ \ell_2 $-approximation is $ w^*=(B^\top B)^{-1}B^\top\ones $, and consequently 
\begin{equation}
\Norm{w^*}_2\le \frac{\Norm{\ones}_2}{\sigma_{\min}(B)} ,
\label{eq:normw}
\end{equation}
  where $ \sigma_{\min}(B) $ denotes the smallest singular value of $ B $.
Let 
\[
\bar B \triangleq [\mathbf{1},B] =
    \begin{pmatrix}
        1&1/M & (1/M)^2 & \cdots & (1/M)^L  \\
        1&2/M & (2/M)^2 & \cdots & (2/M)^L  \\
        1&\vdots  & \vdots  & \ddots & \vdots  \\
        1&1 & 1 & \cdots & 1
    \end{pmatrix}
\]
which is an $M\times (L+1)$ Vandermonde matrix and satisfies $ \sigma_{\min}(\bar B) \leq \sigma_{\min}(B) $ since $\bar B$ has one extra column.
The Gram matrix of $ \bar B $ is an instance of \emph{moment matrices}.
A moment matrix associated with a probability measure $ \mu $ is a Hankel matrix $M$ given by $ M_{i,j}=m_{i+j-2} $, where 
    $m_\ell = \int x^\ell \diff \mu$
denotes the $ \ell $th moment of $ \mu $.
Then $ \frac{1}{M}\bar{B}^\top\bar{B} $ is the moment matrix associated with the uniform distribution over the discrete set $ \{\frac{1}{M},\frac{2}{M},\dots,1\} $, which converges to the uniform distribution over the interval $(0,1)$.
The moment matrix of the uniform distribution is the famous \emph{Hilbert matrix} $H$, with
\[
H_{ij} = \frac{1}{i+j-1}
\]
 which is a well-studied example of ill-conditioned matrices 
in the numerical analysis literature. In particular, it is known that the condition number of the $L\times L$ Hilbert matrix is $ O(\frac{(1+\sqrt{2})^{4L}}{\sqrt{L}}) $ \cite{Todd1954} and the operator norm is $ \Theta(1) $, and thus the minimum singular value is exponentially small in the degree.
Therefore we expect the discrete moment matrix $ \frac{1}{M}\bar{B}^\top\bar{B} $ to behave similarly to the Hilbert matrix when $M$ is large enough.
Interestingly, we show that this is indeed the case as soon as $ M $ exceeds $ L $ (otherwise the minimum singular value is zero).

\begin{lemma}
    \label{lmm:sigma-min}
    For all $ M\ge L+1 $,
    \begin{equation}
        \sigma_{\min}\pth{\frac{\bar B}{\sqrt{M}}}
        \ge \frac{1}{L^22^{7L}(2L+1)}\pth{\frac{M+L}{eM}}^{L+0.5}.
        \label{eq:sigma-min}
    \end{equation}
\end{lemma}
\begin{remark}
    \label{rmk:timan-coeffs-discrete}       
    The inequality \prettyref{eq:timan-coeffs-discrete}     follows from \prettyref{lmm:sigma-min} since the coefficient vector $w=(w_0,\ldots,w_L)$ satisfies
    $\|w\|_\infty \leq \|w\|_2 \leq \frac{1}{\sigma_{\min}(\bar B)} \|\bar B w\|_2 \leq \frac{\sqrt{M}}{\sigma_{\min}(\bar B)} \|\bar B w\|_\infty$.
\end{remark}
\begin{remark}
    The extreme singular values of square Vandermonde matrices have been extensively studied (c.f.~\cite{Gautschi1990,Beckermann2000} and the references therein).
    For rectangular Vandermonde matrices, the focus was mainly with nodes on the unit circle in the complex domain \cite{CGR1990,Ferreira1999,Moitra2015} with applications in signal processing. In contrast, \prettyref{lmm:sigma-min} is on rectangular Vandermonde matrices with real nodes.
    The result on integers nodes in \cite{EPS01} turns out to be too crude for the purpose of this paper.
\end{remark}

\begin{proof}
    Note that $ \bar B^\top\bar B $ is the Gramian of monomials $ \mathbf{x}=(1,x,x^2,\dots,x^L)^\top $ under the inner product defined in \prettyref{eq:inner}.
    When $ M\ge L+1 $, the orthonormal basis $ \phi=(\phi_0,\dots,\phi_L)^\top $ under the inner product \prettyref{eq:inner} are given in \prettyref{eq:def-phi}.
    Let $ \phi=\mathbf{Lx} $ where $ \mathbf{L} \in \reals^{(L+1)\times (L+1)} $ is a lower triangular matrix and $ \bf{L} $ consists of the coefficients of $ \phi $.
    Taking the Gramian of $ \phi $ yields that $ I={\bf L}(\bar B^\top\bar B){\bf L}^\top $, \ie, $ {\bf L}^{-1} $ can be obtained from the Cholesky decomposition: $ \bar B^\top\bar B=({\bf L}^{-1})({\bf L}^{-1})^\top $.    
    Then\footnote{The lower bound \prettyref{eq:lambdamin-L}, which was also obtained in \cite[(1.13)]{CL99} using Cauchy-Schwarz inequality, is tight up to polynomial terms in view of the fact that $ \Norm{\bf L}_{F}\le (L+1)\Norm{\bf L}_{op} $.}
    \begin{equation}
        \sigma_{\min}^2(\bar B)
        =\frac{1}{\norm{{\bf L}}^2_{op}}
        \ge \frac{1}{\norm{{\bf L}}^2_F},
        \label{eq:lambdamin-L}
    \end{equation}
    where $ \norm{\cdot}_{op} $ denotes the $ \ell_2 $ operator norm, which is the largest singular value of $ L $, and $ \norm{\cdot}_F $ denotes the Frobenius norm.
    By definition, $ \norm{{\bf L}}^2_F $ is the sum of all squared coefficients of $ \phi_0,\dots,\phi_L $.
    A useful method to bound the sum-of-squares of the coefficients of a polynomial is by its maximal modulus over the unit circle on the complex plane. Specifically,
    for any polynomial $ p(z)=\sum_{i=0}^{n}a_iz^i $, we have 
    \begin{equation}
        \sum_{i=0}^{n}|a_i|^2=\frac{1}{2\pi}\oint_{|z|=1}|p(z)|^2\diff z \le \sup_{|z|=1}| p(z)|^2 .
        \label{eq:sos}
    \end{equation}
    Therefore
    \begin{equation}
        \sigma_{\min}(\bar B)
        \ge \frac{1}{\norm{{\bf L}}_F}
        \ge \frac{1}{\sqrt{\sum_{m=0}^{L}\sup_{|z|=1}|\phi_m(z)|^2}}
        \ge \frac{1}{\sqrt{L+1}}\frac{1}{\sup_{0\leq m \leq L,|z|=1}|\phi_m(z)|}.
        \label{eq:coeffs-phiz}
    \end{equation}

    For a given $ M $, the orthonormal basis $ \phi_m(x) $ in \prettyref{eq:def-phi} is proportional to the discrete Chebyshev polynomials $ t_m(Mx-1) $.
    The classical asymptotic result for the discrete Chebyshev polynomials shows that \cite[(2.8.6)]{orthogonal.poly}
    \begin{equation*}
        \lim_{M\rightarrow \infty}M^{-m}t_m(Mx)=P_m(2x-1),
    \end{equation*}
    where $ P_m $ is the Legendre polynomial of degree $ m $.
    This gives the intuition that $ t_m(x)\approx M^m $ for real-valued $ x\in [0,M] $.
    We have the following non-asymptotic upper bound (proved in \prettyref{app:aux}) for $ t_m $ over the complex plane: 
    \begin{lemma}
        \label{lmm:tm-ub}
        For all $ 0\le m\le M-1 $,
        \begin{equation}
            |t_m(z)|\le m^22^{6m}\sup_{0\le \xi \le m}\pth{|z+\xi|\vee M}^m.
            \label{eq:tm-ub}
        \end{equation}
    \end{lemma}

    Applying \prettyref{eq:tm-ub} on the definition of $ \phi_m $ in \prettyref{eq:def-phi}, for any $ |z|=1 $ and any $ M\ge L+1 $, we have
    \begin{equation*}
        |\phi_m(z)|=\frac{|t_m(M z-1)|}{\sqrt{c(M,m)}}
        \le \frac{m^22^{7m}M^m}{\sqrt{\frac{M(M^2-1^2)(M^2-2^2)\cdots(M^2-m^2)}{2m+1}}}.
    \end{equation*}
    The right-hand side is increasing with $ m $.
    Therefore,
    \begin{align*}
      \sup_{0\le m\le L, |z|=1}|\phi_m(z)|
      &\le \frac{L^22^{7L}M^L}{\sqrt{\frac{M(M^2-1^2)(M^2-2^2)\cdots(M^2-L^2)}{2L+1}}}\\
      &= \frac{1}{\sqrt{M}}L^22^{7L}\sqrt{2L+1}\sqrt{\frac{M^{2L+1}}{\binom{M+L}{2L+1}(2L+1)!}}.
    \end{align*}
    Combining \prettyref{eq:coeffs-phiz}, we obtain
    \begin{align*}
      \sigma_{\min}\pth{\frac{\bar B}{\sqrt{M}}}
      &\ge \frac{1}{L^22^{7L}\sqrt{(L+1)(2L+1)}}\sqrt{\frac{\binom{M+L}{2L+1}(2L+1)!}{M^{2L+1}}}\\
      &\ge \frac{1}{L^22^{7L}(2L+1)}\pth{\frac{M+L}{eM}}^{L+0.5},
    \end{align*}
    where in the last inequality we used $ \binom{n}{k}\ge (\frac{n}{k})^k $ and $ n!\ge (\frac{n}{e})^n $.
\end{proof}

Using the optimal solution $ w^* $ to the $\ell_2$-approximation problem \prettyref{eq:L2} as the coefficient of the linear estimator $ \hat C $, 
the following performance guarantee is obtained by applying \prettyref{lmm:lse-min} and  \prettyref{lmm:sigma-min} to bound the bias and variance, respectively:
\begin{theorem}
    \label{thm:hatC-lse}
    Assume the Poisson sampling model.
    Then,
    \begin{equation}
        \Expect{(\hat C-C)^2}
        \le  k^2\exp\pth{-\Theta\pth{1\vee \frac{n\log k}{k}\wedge \log k}}.
        \label{eq:hatC-lse}
    \end{equation}
\end{theorem}
\begin{proof}
    If $ n \le \frac{k}{\log k} $, then the upper bound in \prettyref{eq:hatC-lse} is $ \Theta(k^2) $, which is trivial thanks to the thresholds that $ \hat{C}=(\tilde{C}\vee \Cplug)\wedge k $.
    It is hereinafter assumed that $ n\ge \frac{k}{\log k} $, or equivalently $ M \le \frac{\beta}{\alpha^2}L^2 $; here $M,L$ are defined in \prettyref{eq:LM} and the constants $\alpha,\beta$ are to be determined later.
    Then, from \prettyref{lmm:lse-min}, 
    \begin{equation}
        \Norm{Bw^*-\ones}_\infty\le \Norm{Bw^*-\ones}_2
        \le \exp\pth{-\Theta\pth{\frac{L^2}{M}}}.
        \label{eq:hatC-lse-ref1}
    \end{equation}
    In view of \prettyref{eq:normw} and \prettyref{lmm:sigma-min}, we have
    \begin{equation*}
        \norm{w^*}_\infty\le \Norm{w^*}_2\le \frac{\Norm{\ones}_2}{\sigma_{\min}(B)}\le \exp(O(L)).
    \end{equation*}
    Recall the connection between $ u_j $ and $ w_j $ in \prettyref{eq:Bw}.
    For $ 1\le j\le L < \beta \log k $, we have $ u_j=w_j\frac{j!}{(\beta \log k)^j}\le \frac{w_j}{\beta \log k} $.
    Therefore,
    \begin{equation}
        \Norm{u^*}_\infty\le \frac{\Norm{w^*}_\infty}{\beta\log k}\le \frac{\exp(O(L))}{\beta\log k}.
        \label{eq:hatC-lse-ref2}
    \end{equation}
    Applying \prettyref{eq:hatC-lse-ref1} and \prettyref{eq:hatC-lse-ref2} to \prettyref{prop:rate-w}, we obtain
    \begin{equation*}
        \Expect{(\hat C-C)^2}
        \le k^2\exp\pth{-\frac{2n}{k}-\Theta\pth{\frac{n\log k}{k}}}+ke^{-n/k}+k\frac{\exp(O(\log k))}{(\beta\log k)^2}+k^{-(\beta-\alpha\log\frac{e\beta}{\alpha}-3)}.
    \end{equation*}
    Then the desired \prettyref{eq:hatC-lse} holds as long as $ \beta $ is sufficiently large and $ \alpha  $ is sufficiently small.
\end{proof}

\subsection{Lagrange interpolating polynomials and Stirling numbers}
\label{sec:lag}
When we sample at least a constant faction of the urn, i.e., $ n=\Omega(k) $, we can afford to choose $\alpha$ and $\beta$ in \prettyref{eq:LM} so that $L=M$ and $B$ is an invertible matrix. We choose the coefficient $ w=B^{-1}\ones $ which is equivalent to applying \emph{Lagrange interpolating polynomial} and achieves exact zero bias.
To control the variance, we can follow the approach in \prettyref{sec:sigma-min} by using the bound on minimum singular value of the matrix $ B $, which implies that the coefficients are $\exp(O(L))$ and yields a coarse upper bound $ O(k\frac{\log k}{1\vee \log\frac{\Delta^2}{k}}) $ on the sample complexity.
As previously announced in \prettyref{tab:main}, this bound can be improved to $ O(k\log \frac{\log k}{1\vee \log\frac{\Delta^2}{k}}) $ by a more careful analysis of the Lagrange interpolating polynomial coefficients expressed in terms of the Stirling numbers,  which we introduce next.


The Stirling numbers of the first kind are defined as the coefficients of the  falling factorial $ (x)_n $ where
\begin{equation*}
    (x)_n=x(x-1)\dots(x-n+1)=\sum_{j=1}^{n}s(n,j)x^j.
\end{equation*}
Compared to the coefficients $ w $ expressed by the Lagrange interpolating polynomial:
\begin{equation*}
    \sum_{j=1}^{M}w_jx^j-1=-\frac{(1-xM)(2-xM)\dots(M-xM)}{M!},
\end{equation*}
we obtain a formula for the coefficients $ w $ in terms of the Stirling numbers:
\begin{equation*}
    w_j=\frac{(-1)^{M+1}M^j}{M!}s(M+1,j+1),\quad 1\le j \le M.
\end{equation*}
Consequently, the coefficients of our estimator $ u_j $ are given by
\begin{equation}
    u_j=(-1)^{M+1}\frac{j!}{M!}\pth{\frac{k}{n}}^js(M+1,j+1).
    \label{eq:uj-stirling}
\end{equation}
The precise asymptotics the Stirling number is rather complicated. In particular, the asymptotic formula of $ s(n,m) $ as $ n\rightarrow\infty $ for fixed $ m $ is given by \cite{Jordan47} and the uniform asymptotics over all $ m $ is obtained in \cite{MW58} and \cite{Temme93}.
The following lemma (proved in \prettyref{app:aux}) is a coarse non-asymptotic version, which suffices for the purpose of constant-factor approximations of the sample complexity.
\begin{lemma}
    \label{lmm:stirling}
    \begin{equation}
        |s(n+1,m+1)|=n!\pth{\Theta\pth{\frac{1}{m}\pth{1\vee \log\frac{n}{m}}}}^m
        \label{eq:stirling}
    \end{equation}
\end{lemma}

We construct $ \hat C $ as in \prettyref{prop:rate-w} using the coefficients $ u_j $ in \prettyref{eq:uj-stirling} to achieve zero bias.
The variance upper bound by the coefficients $ u $ is a direct consequence of the upper bound of Stirling numbers in \prettyref{lmm:stirling}.
Then we obtain the following mean squared error (MSE):
\begin{theorem}[Interpolation]
    \label{thm:hatC-interpolation}
    Assume the Poisson sampling model.
    If $ n> \eta k $ for some sufficiently large constant $ \eta $, then
    \begin{equation*}
        \Expect(\hat C-C)^2
        \le ke^{-\Theta(\frac{n}{k})}+ k^{-0.5-3.5\frac{k}{n}\log\frac{k}{en}}+
        \begin{cases}
            k\exp\pth{\frac{k^2\log k}{n^2} e^{-\Theta(\frac{n}{k})}},& n \lesssim k \log\log k,\\
            k\pth{\Theta\pth{\frac{k}{n}}\log\frac{k^2\log k}{n^2}}^{2n/k},& k\log\log k\lesssim n\lesssim k\sqrt{\log k},\\
            0,& n \gtrsim k\sqrt{\log k}.
        \end{cases}
    \end{equation*}
\end{theorem}
\begin{proof}
    In \prettyref{prop:rate-w}, fix $ \beta =3.5 $ and $ \alpha =\frac{\beta k}{n} $ so that $L=M$.
    Our goal is to show an upper bound of
    \begin{equation}
        \max_{\lambda\in\frac{n}{k}[M]}\Expect_{N\sim\Poi(\lambda)}[u_{N}^2]
        =\max_{\lambda\in\frac{n}{k}[M]}\sum_{j=1}^Mu_j^2e^{-\lambda}\frac{\lambda^j}{j!}.
        \label{eq:sum-ref}
    \end{equation}
    Here the coefficients $ u_j $ are obtained from \prettyref{eq:uj-stirling} and, in view of \prettyref{eq:stirling}, satisfy:
    \begin{equation}
        |u_j|\le \pth{\frac{\eta k}{n}\pth{1\vee \log\frac{M}{j}}}^j,\quad 1\le j \le M,
        \label{eq:uj-ref}
    \end{equation}
    for some universal constant $\eta$.
    We consider three cases separately:
    
    \paragraph{Case I: $ n\ge \sqrt{\beta} k\sqrt{\log k} $.}
    In this case we have $ \frac{n}{k}\ge M $.
    The maximum of each summand in \prettyref{eq:sum-ref} as a function of $ \lambda\in\reals $ occurs at $ \lambda=j $.
    Since $ j\le \frac{n}{k} $, the maximum over $ \lambda\in\frac{n}{k}[M] $ is attained at $ \lambda=\frac{n}{k} $.
    Then,
    \begin{equation}
        \max_{\lambda\in\frac{n}{k}[M]}\Expect_{N\sim\Poi(\lambda)}[u_{N}^2]
        = \Expect_{N\sim\Poi(\frac{n}{k})}[u_{N}^2].
        \label{eq:range1-ref}
    \end{equation}
    In view of \prettyref{eq:uj-ref} and $ j\ge 1 $, we have $ |u_j|\le (\Theta(k/n)\log M)^j $.
    Then,
    \begin{align*}
      \Expect_{N\sim\Poi(\frac{n}{k})}[u_{N}^2]
      &\le \Expect_{N\sim\Poi(\frac{n}{k})}\pth{\Theta\pth{\frac{k\log M}{n}}^2}^{N}\\
      &= \exp\pth{\frac{n}{k}\pth{\Theta\pth{\frac{k\log M}{n}}^{2}-1}}
        = e^{-\Theta(n/k)},
    \end{align*}
    as long as $ n\gtrsim k\log\log k $ and thus $ \frac{k\log M}{n}\lesssim 1 $.
    Therefore,
    \begin{equation}
        \max_{\lambda\in\frac{n}{k}[M]}\Expect_{N\sim\Poi(\lambda)}[u_{N}^2]
        \le e^{-\Theta(n/k)},\quad n\gtrsim k\sqrt{\log k}.
        \label{eq:Eub-range1}
    \end{equation}

    \paragraph{Case II: $ \eta k\log\log k\le n\le \sqrt{\beta}k\sqrt{\log k} $.}
    We apply the following upper bound:
    \begin{align}
      \max_{\lambda\in\frac{n}{k}[M]}\Expect_{N\sim\Poi(\lambda)}[u_{N}^2]
      =&\max_{\lambda\in\frac{n}{k}[M]}\Expect_{N\sim\Poi(\lambda)}[u_{N}^2\indc{N\ge n/k}]+\max_{\lambda\in\frac{n}{k}[M]}\Expect_{N\sim\Poi(\lambda)}[u_{N}^2\indc{N< n/k}]\nonumber\\
      \le &\max_{\frac{n}{k}\le j\le M}|u_j|^2+e^{-\Theta(n/k)}.\label{eq:range2-ref}
    \end{align}
    where the upper bound of the second addend is analogous to \prettyref{eq:range1-ref} and \prettyref{eq:Eub-range1}.
    Since $ \frac{\eta k}{n}\le 1 $, the right-hand side of \prettyref{eq:uj-ref} is decreasing with $ j $ when $ j\ge M/e $.
    It suffices to consider $ j\le M/e $, when the maximum as a function of $ j\in\reals $ occurs at $ j^*\le Me^{-\frac{n}{\eta k}} $.
    Since $ Me^{-\frac{n}{\eta k}}\le \frac{n}{k} $ when $ n\ge \eta k\log\log k $, the maximum over $ \frac{n}{k}\le j\le M $ is attained at $ j=\frac{n}{k} $.
    Applying \prettyref{eq:uj-ref} with $ j=\frac{n}{k} $ to \prettyref{eq:range2-ref} yields
    \begin{equation}
        \max_{\lambda\in\frac{n}{k}[M]}\Expect_{N\sim\Poi(\lambda)}[u_{N}^2]
        \le \pth{\Theta\pth{\frac{k}{n}}\log\frac{k^2\log k}{n^2}}^{2n/k}+e^{-\Theta(n/k)}.
        \label{eq:Eub-range2}
    \end{equation}

    \paragraph{Case III: $ \eta k\le n\le \eta k\log\log k $.}
    We apply the upper bound of expectation by the maximum:
    \begin{equation*}
        \max_{\lambda\in\frac{n}{k}[M]}\Expect_{N\sim\Poi(\lambda)}[u_{N}^2]
        \le \max_{j\in[M]}u_j^2.
    \end{equation*}
    Since $ \frac{\eta k}{n}\le 1 $, the right-hand side of \prettyref{eq:uj-ref} is decreasing with $ j $ when $ j\ge M/e $, so it suffices to consider $ j\le M/e $.
    Denoting $ x=\log\frac{M}{j} $ and $ \tau=\Theta(\frac{k}{n}) $, in view of \prettyref{eq:uj-ref}, we have $ |u_j|\le \exp(Me^{-x}\log(\tau x)) $, which attains maximum at $ x^* $ satisfying $ \frac{e^{1/x^*}}{x^*}=\tau $.
    Then, 
    \begin{equation*}
        |u_j|\le \exp(Me^{-x^*}\log(\tau x^*))=\exp(Me^{-x^*}/x^*)
        < \exp(M \tau e^{-1/\tau}).
    \end{equation*}
    where the last inequality is because of $ \tau>\frac{1}{x^*} $.
    Therefore,
    \begin{equation}
        \max_{\lambda\in\frac{n}{k}[M]}\Expect_{N\sim\Poi(\lambda)}[u_{N}^2]
        \le \exp\pth{\frac{k^2\log k}{n^2} e^{-\Theta(\frac{n}{k})}},\quad k\lesssim n\lesssim k\log\log k.
        \label{eq:Eub-range3}
    \end{equation}

    Applying the upper bounds in \prettyref{eq:Eub-range1}, \prettyref{eq:Eub-range2} and \prettyref{eq:Eub-range3} to \prettyref{prop:rate-w} concludes the proof.
\end{proof}

\begin{remark}
    It is impossible to bridge the gap near $ \Delta=\sqrt{k} $ in \prettyref{tab:main} using the technology of interpolating polynomials that aims at zero bias, since its worst-case variance is at least $ k^{1+\Omega(1)} $ when $ n=O(k) $. To see this, note that the variance term given  by \prettyref{eq:calE-var} is
    \begin{equation}
        \sum_{p_i}\Expect_{N\sim\Poi(np_i)}[u_N^2]
        =\sum_{p_i}\sum_{j=1}^Lu_j^2e^{-np_i}\frac{(np_i)^j}{j!}.
        \label{eq:var-sum}
    \end{equation}
    Consider the distribution $ \Uniform[n/j_0] $ with $ j_0=L e^{-2n/k}=\Omega(\log k) $, which corresponds to an urn where each of the $n/j_0$ colors appears equal number of times.
    By the formula of coefficient $ u_j $ in \prettyref{eq:uj-stirling} and the characterization from \prettyref{lmm:stirling}, the $ j=j_0 $ term in the summation of \prettyref{eq:var-sum} is of order $ \frac{n}{j_0}(\frac{k}{n}\log\frac{M}{j_0})^{2j_0}=\frac{n}{j_0}2^{2j_0} $, which is already $ k^{1+\Omega(1)} $.
\end{remark}

\section{Optimality of the sample complexity}
\label{sec:lb}
In this section we develop lower bounds of the sample complexity which certify the optimality of estimators constructed in \prettyref{sec:ub}.
We first give a brief overview of the lower bound in \cite[Theorem 1]{CCMN00}, which gives the optimal sample complexity under the multiplicative error criterion.
The lower bound argument boils down to considering two hypothesis: in the null hypothesis, the urn consists of only one color;
in the alternative, the urn contains $ 2\Delta+1 $ distinct colors, where $ k-2\Delta $ balls share the same color as in the null hypothesis, and all other balls have distinct colors.
These two scenarios are distinguished if and only if a second color appears in the samples, which typically requires $ \Omega(k/\Delta) $ samples.
This lower bound is optimal for estimating within a multiplicative factor of $ \sqrt{\Delta} $, which, however, is too loose for additive error $ \Delta $.

In contrast, instead of testing whether the urn is monochromatic, our first lower bound is given by testing whether the urn is maximally colorful, that is, containing $ k $ distinct colors.
The alternative contains $ k-2\Delta $ colors, and the numbers of balls of two different colors differ by at most one.
In other words, the null hypothesis is the uniform distribution on $ [k] $ and the alternative is close to uniform distribution with smaller support size.
The sample complexity, which is shown in \prettyref{thm:main-lb-uniform}, gives the lower bound in \prettyref{tab:main} for $ \Delta\le \sqrt{k} $.
\begin{theorem}
    \label{thm:main-lb-uniform}
    If $ 1\le \Delta \le \frac{k}{2} $, then
    \begin{equation}
        n^*(k,\Delta) \ge \Omega\pth{\frac{k-2\Delta}{\sqrt{k}}}.
        \label{eq:uniform-lb1}
    \end{equation}
    If $ 1\le \Delta < \frac{k}{4} $, then
    \begin{equation}
        n^*(k,\Delta) \ge \Omega\pth{k\arccosh\pth{1+\frac{k}{4\Delta^2}}}
        \asymp
        \begin{cases}
            k \log (1+\frac{k}{\Delta^2}), & \Delta \leq \sqrt{k}, \\
            \frac{k^{3/2}}{\Delta}, & \Delta \geq \sqrt{k}.
        \end{cases}
        \label{eq:uniform-lb2}
    \end{equation}
\end{theorem}
\begin{proof}
    Consider the following two hypotheses:
    The null hypothesis $ H_0 $ is an urn consisting of $ k $ distinct colors;
    The alternative $ H_1 $ consists of $ k-2\Delta $ distinct colors, and each color appears either $ b_1\triangleq \Floor{\frac{k}{k-2\Delta}} $ or $ b_2\triangleq \Ceil{\frac{k}{k-2\Delta}} $ times.
    In terms of distributions, $ H_0 $ is the uniform distribution $ Q=(\frac{1}{k},\dots,\frac{1}{k}) $;
    $ H_1 $ is the closest perturbation from the uniform distribution: randomly pick disjoint sets of indices $ I,J\subseteq [k] $ with cardinality $ |I|=c_1 $ and $ |J|=c_2 $, where $ c_1 $ and $ c_2 $ satisfy
    \begin{align*}
      \text{(number of colors)}\quad & c_1+c_2=k-2\Delta, \\
      \text{(number of balls)}\quad & c_1b_1+c_2b_2=k.
    \end{align*}
    Conditional on $ \theta\triangleq (I,J) $, the distribution $ P_{\theta}=(p_{\theta,1},\dots,p_{\theta,k}) $ is given by
    \begin{equation*}
        p_\theta=
        \begin{cases}
            b_1/k,& i\in I, \\
            b_2/k,& i\in J.
        \end{cases}
    \end{equation*}
    Put the uniform prior on the alternative.
    Denote the marginal distributions of the $ n $ samples $ X=(X_1,\dots,X_n) $ under $ H_0 $ and $ H_1 $ by $ Q_{X} $ and $ P_{X} $, respectively.
    Since the distinct colors in $ H_0 $ and $ H_1 $ are separated by $ 2\Delta $, to show that the sample complexity $ n^*(k,\Delta)\ge n $, it suffices to show that no test can  distinguish $ H_0 $ and $ H_1 $ reliably using $ n $ samples.
    A further sufficient condition is a bounded $ \chi^2 $ divergence \cite{Tsybakov09}
    \begin{equation*}
        \chi^2(P_{X}\|Q_{X})\triangleq\int \frac{P_{X}^2}{Q_{X}}-1\le O(1).
    \end{equation*}
    The remainder of this proof is devoted to upper bounds of the $ \chi^2 $ divergence.
    
    Since $ P_{X|\theta}=P_{\theta}^{\otimes n} $ and $ Q_{X}=Q^{\otimes n} $, we have
    \begin{align*}
      \chi^2(P_{X}\|Q_{X})+1
      &=\int \frac{P_{X}^2}{Q_{X}}
      =\int \frac{(\Expect_\theta P_{X|\theta}) (\Expect_{\theta'} P_{X|\theta'}) }{Q_{X}}\\
      &=\Expect_{\theta,\theta'}\int \frac{P_{X|\theta} P_{X|\theta'}}{Q_{X}}
      =\Expect_{\theta,\theta'}\pth{\int \frac{P_{\theta} P_{\theta'}}{Q}}^n,
    \end{align*}
    where $ \theta' $ is an independent copy of $ \theta $.
    By the definition of $ P_{\theta} $ and $ Q $,
    \begin{equation}
        \int \frac{P_{\theta} P_{\theta'}}{Q}
        = \frac{b_1^2}{k}|I\cap I'|+\frac{b_2^2}{k}|J\cap J'|+\frac{b_1b_2}{k}(|I\cap J'|+|J\cap I'|)
        = 1+\sum_{i=1}^{4}A_i,
        \label{eq:chi2-ref4}
    \end{equation}
    where $ A_1\triangleq \frac{b_1^2}{k}(|I\cap I'|-\frac{c_1^2}{k}) $, $ A_2\triangleq \frac{b_2^2}{k}(|J\cap J'|-\frac{c_2^2}{k}) $, $ A_3=\frac{b_1b_2}{k}(|I\cap J'|-\frac{c_1c_2}{k}) $, and $ A_4=\frac{b_1b_2}{k}(|J\cap I'|-\frac{c_1c_2}{k}) $ are centered random variables.
    Applying $ 1+x\le e^x $ and Cauchy-Schwarz inequality, we obtain
    \begin{equation}
        \chi^2(P_{X}\|Q_{X})+1
        \le \Expect[e^{n\sum_{i=1}^{4}A_i}]
        \le \prod_{i=1}^4(\Expect[e^{4nA_i}])^{\frac{1}{4}}.
        \label{eq:chi2-ref1}
    \end{equation}
    Consider the first term $ \Expect[e^{4nA_1}] $.
    Note that $ |I\cap I'|\sim \Hypergeo(k,c_1,c_1) $, which is the distribution of the sum of $ c_1 $ samples drawn without replacement from a population of size $ k $ which consists of $ c_1 $ ones and $ k-c_1 $ zeros.
    By the convex stochastic dominance of the binomial over the hypergeometric distribution \cite[Theorem 4]{Hoeffding63}, for $ Y\sim\Binom(c_1,\frac{c_1}{k}) $, we have
    \begin{align}
      (\Expect[e^{4nA_1}])^{\frac{1}{4}}
      & \le \pth{\expect{\exp\pth{\frac{4nb_1^2}{k}(Y-c_1^2/k)}}}^{\frac{1}{4}}\nonumber\\
      & \le \exp\pth{\frac{c_1^2}{4k}\pth{\exp\pth{\frac{4nb_1^2}{k}}-1-\frac{4nb_1^2}{k}}}\nonumber\\
      & \le \exp\pth{\frac{c_1^2}{4k}\pth{\exp\pth{\frac{4nb_2^2}{k}}-1-\frac{4nb_2^2}{k}}},\label{eq:chi2-ref2}
    \end{align}
    where the last inequality follows from the fact that $ x\mapsto e^x-1-x $ is increasing when $ x>0 $.
                Other terms in \prettyref{eq:chi2-ref1} are bounded analogously and we have
    \begin{align}
      \chi^2(P_{X}\|Q_{X})+1
      & \le \exp\pth{\frac{c_1^2+c_2^2+2c_1c_2}{4k}\pth{\exp\pth{\frac{4nb_2^2}{k}}-1-\frac{4nb_2^2}{k}}}\nonumber\\
      & = \exp\pth{\frac{(k-2\Delta)^2}{4k}\pth{\exp\pth{\frac{4n}{k}\ceil{\frac{k}{k-2\Delta}}^2}-1-\frac{4n}{k}\ceil{\frac{k}{k-2\Delta}}^2}}.\label{eq:chi2-ref3}
    \end{align}
    If $ k-2\Delta\ge \sqrt{k} $, the upper bound \prettyref{eq:chi2-ref3} implies that $ n^*(k,\Delta)\ge \Omega(\frac{k-2\Delta}{\sqrt{k}}) $ since the $\chi^2$-divergence is finite with $O(\frac{k-2\Delta}{\sqrt{k}})$ samples, using the inequality that $e^x-1-x\le \frac{x^2}{2}$ for $x\ge 0$;
    if $ k-2\Delta\le \sqrt{k} $, the lower bound is trivial since $ \frac{k-2\Delta}{\sqrt{k}}\le 1 $.

    Now we prove the refined estimate \prettyref{eq:uniform-lb2} for $ 1\le \Delta < k/4 $, in which case $ |I|=c_1=k-4\Delta, |J|=c_2=2\Delta $ and $ b_1=1,b_2=2 $.
When $ c_1 $ is close to $ k $, $\Hypergeo(k,c_1,c_1) $ is no longer well approximated by $ \Binom(c_1,\frac{c_1}{k}) $,
and the upper bound in \prettyref{eq:chi2-ref2} yields a loose lower bound for the sample complexity.
    To fix this, note that in this case the set $ K\triangleq (I\cup J)^c $ has small cardinality $ |K|=2\Delta $.
    The equality in \prettyref{eq:chi2-ref4} can be equivalently represented in terms of $ J,J' $ and $ K,K' $ by
    \begin{equation*}
        \int \frac{P_{\theta} P_{\theta'}}{Q}
        = 1+\frac{|J\cap J'|+|K\cap K'|-|J\cap K'|-|K\cap J'|}{k}.
    \end{equation*}
    By upper bounds analogous to \prettyref{eq:chi2-ref1} -- \prettyref{eq:chi2-ref3}, $ \chi^2(P_{X}\|Q_{X})+1\le \prod_{i=1}^4(\Expect[e^{4nB_i}])^{\frac{1}{4}} $, where $ B_1\triangleq \frac{1}{k}(|J \cap J'|-\frac{(2\Delta)^2}{k}) $, $ B_2\triangleq \frac{1}{k}(|K \cap K'|-\frac{(2\Delta)^2}{k}) $, $ B_3\triangleq -\frac{1}{k}(|J \cap K'|-\frac{(2\Delta)^2}{k}) $, and $ B_4\triangleq -\frac{1}{k}(|K \cap J'|-\frac{(2\Delta)^2}{k}) $.
    Note that $|J \cap J'|,|K \cap K'|,|J \cap K'|,|K \cap J'|$ are all distributed as $\Hypergeo(k,2\Delta,2\Delta)$, which is dominated by $\Binom(2\Delta,\frac{2\Delta}{k})$.
    For $ Y\sim\Binom(2\Delta,\frac{2\Delta}{k}) $, we have
    \begin{equation*}
        (\Expect[e^{4nB_i}])^{\frac{1}{4}}
        \le \pth{\expect{\exp\pth{t\pth{Y-\frac{(2\Delta)^2}{k}}}}}^{1/4}
        \le \exp\pth{\frac{(2\Delta)^2}{4k}\pth{e^t-1-t}}.
    \end{equation*}
    with $ t =\frac{4n}{k} $ for $ i=1,2 $ and $ t =-\frac{4n}{k} $ for $ i=3,4 $.
    Therefore,
    \begin{align}
        \chi^2(P_{X}\|Q_{X})+1
        &\le \exp\pth{\frac{\Delta^2}{k}\pth{2e^{4n/k}+2e^{-4n/k}-4}}\nonumber\\
        &= \exp\pth{\frac{4\Delta^2}{k}(\cosh(4n/k)-1)}.
        \label{eq:chi2-ref5}
    \end{align}
    The upper bound \prettyref{eq:chi2-ref5} yields the sample complexity $ n^*(k,\Delta)\ge \Omega(k\arccosh(1+\frac{k}{4\Delta^2})) $.
\end{proof}

Now we establish another lower bound for the sample complexity of the \DistinctElements problem for sampling without replacement.
Since we can simulate sampling with replacement from samples obtained without replacement (see \prettyref{eq:hyper-multi} for details), it is also a valid lower bound for $ n^*(k,\Delta) $ defined in \prettyref{def:sample}.
On the other hand, as observed in \cite[Lemma 3.3]{RRSS09} (see also \cite[Lemma 5.14]{GV-thesis}), any estimator $ \hat{C} $ for the \DistinctElements problem with sampling without replacement leads to an estimator for the \SupportSize problem with slightly worse performance:
Suppose we have $ n $ \iid~samples drawn from a distribution $ P $ whose minimum non-zero probability is at least $ 1/\ell $.
Let $ \Cplug $ denote the number of distinct elements in these samples.
Equivalently, these samples can be viewed as being generated in two steps:
first, we draw $ k $ \iid~samples from $ P $, whose realizations form an instance of a $ k $-ball urn with $ \Cplug $ distinct colors;
next, we draw $ n $ samples from this urn without replacement ($ n \leq k $), which clearly are distributed according to $ P^{\otimes n} $.
Suppose $ \Cplug $ is close to the actual support size of $ P $.
Then applying any algorithm for the \DistinctElements problem to these $ n $ \iid~samples constitutes a good support size estimator. 
\prettyref{lmm:w-wo-replacement} formalizes this intuition.
\begin{lemma}
    \label{lmm:w-wo-replacement}
    Suppose an estimator $ \hat{C} $ takes $ n $ samples from a $ k $-ball urn $ (n\le k) $ without replacement and provides an estimation error of less than $ \Delta $ with probability at least $ 1-\delta $.
    Applying $ \hat{C} $ with $ n $ \iid~samples from any distribution $ P $ with minimum non-zero mass $ 1/\ell $ and support size $ S(P) $, we have
    \begin{equation*}
        |\hat{C}-S(P)|\le 2\Delta
    \end{equation*}
    with probability at least $ 1-\delta-\binom{\ell}{\Delta}\pth{1-\frac{\Delta}{\ell}}^k $.
\end{lemma}
\begin{proof}
    Suppose that we take $ k $ \iid~samples from $ P=(p_1,p_2,\dots) $, which form a $ k $-ball urn consisting of $ C $ distinct colors.
    By the union bound,
    \begin{equation*}
        \Prob[|C-S(P)|\ge \Delta]
        \le \sum_{\substack{I:|I|=\Delta,\\p_i\ge \frac{1}{\ell}, i\in I}}\pth{1-\sum_{i\in I}p_i}^k
        \le \binom{\ell}{\Delta}\pth{1-\frac{\Delta}{\ell}}^k.
    \end{equation*}
    Next we take $ n $ samples without replacement from this urn and apply the given estimator $ \hat{C} $.
    By assumption, conditioned on any realization of the $ k $-ball urn, $ |\hat{C}-C|\le \Delta $ with probability at least $ 1-\delta $.
    Then $ |\hat{C}-S(P)|\le 2\Delta $ with probability at least $ 1-\delta-\binom{\ell}{\Delta}\pth{1-\frac{\Delta}{\ell}}^k $.
    Marginally, these $ n $ samples are identically distributed as $ n $ \iid~samples from $ P $.
\end{proof}

Combining with the sample complexity of the \SupportSize problem in \prettyref{eq:sample-supp}, \prettyref{lmm:w-wo-replacement} leads to the following lower bound for the \DistinctElements problem:
\begin{theorem}
    \label{thm:nstar-support}
    Fix a sufficiently small constant $ c $. For any $ 1\le\Delta\le ck  $,
    \begin{equation*}
        n^*(k,\Delta)\ge \Omega\pth{\frac{k}{\log k}\log\frac{k}{\Delta}}.
    \end{equation*}
    The same lower bound holds for sampling without replacement.
\end{theorem}
\begin{proof}
    By the lower bound of the support size estimation problem obtained in \cite[Theorem 2]{WY15}, if $ n\le \frac{\alpha \ell}{\log \ell}\log^2\frac{\ell}{2\Delta} $ and $ 2\Delta\le c_0\ell $ for some fixed constants $ c_0<\frac{1}{2} $ and $ \alpha $, then for any $ \hat{C} $, there exists a distribution $ P $ with minimum non-zero mass $ 1/\ell $ such that $ |\hat{C}-S(P)|\le 2\Delta $ with probability at most $ 0.8 $.
    Applying \prettyref{lmm:w-wo-replacement} yields that, using $ n $ samples without replacement, no estimator can provide an estimation error of $ \Delta $ with probability $ 0.9 $ for an arbitrary $ k $-ball urn, provided $ \binom{\ell}{\Delta}\pth{1-\frac{\Delta}{\ell}}^k\le 0.1 $.
    Consequently, as long as $ 2\Delta\le c_0\ell $ and $ \binom{\ell}{\Delta}\pth{1-\frac{\Delta}{\ell}}^k\le 0.1 $, we have
    \begin{equation*}
        n^*(k,\Delta)\ge \frac{\alpha \ell}{\log \ell}\log^2\frac{\ell}{2\Delta}.
    \end{equation*}
    The desired lower bound follows from choosing $ \ell\asymp \frac{k}{\log(k/\Delta)} $.
\end{proof}



\section{Proof of results in \prettyref{tab:main}}
\label{sec:table}
Below we explain how the sample complexity bounds summarized in \prettyref{tab:main} are obtained from various results in \prettyref{sec:ub} and \prettyref{sec:lb}:  
\begin{itemize}
        \item The upper bounds are obtained from the worst-case MSE in \prettyref{sec:ub} and the Markov inequality.
In particular, the case of $ \Delta\le \sqrt{k}(\log k)^{-\delta} $ follows from the second and the third upper bounds of \prettyref{thm:hatC-interpolation};
the case of $ \sqrt{k}\le \Delta \le k^{0.5+\delta} $ follows from the first upper bound of \prettyref{thm:hatC-interpolation};
the case of $ k^{1-\delta}\le \Delta\le ck $ follows from \prettyref{thm:hatC-lse}.
By monotonicity, we have the $ O(k\log\log k) $ upper bound when $ \sqrt{k}(\log k)^{-\delta}\le \Delta\le \sqrt{k} $, the $ O(\frac{k}{\log k}) $ upper bound when $ \Delta\ge ck $, and the $ O(k) $ upper bound when $ k^{0.5+\delta}\le \Delta \le k^{1-\delta} $.

\item The lower bound for $ \Delta\le \sqrt{k} $ follows from \prettyref{thm:main-lb-uniform};
the lower bound for $ k^{0.5+\delta}\le \Delta\le ck $ follows from \prettyref{thm:nstar-support}.
These further implies the $ \Omega(k) $ lower bound for $ \sqrt{k}\le \Delta\le k^{0.5+\delta} $ by monotonicity.

\end{itemize}

\appendix
\section{Connections between various sampling models}
\label{app:model}
As mentioned in \prettyref{sec:related}, four popular sampling models have been introduced in the statistics literature: the multinomial model, the hypergeometric model, the Bernoulli model, and the Poisson model.
The connections between those models are explained in details in this section, as well as relations between the respective sample complexities.

\begin{figure}[h]
    \centering
    \begin{tikzpicture}[
            trans/.style={thick,->,shorten >=2pt,shorten <=2pt,>=stealth},
            ]
        \filldraw (0,2) circle (0pt) node[above] (B) {\bf Bernoulli model};
        \filldraw (0,0) circle (0pt) node[below] (H) {\bf hypergeometric model};
        \filldraw (6,0) circle (0pt) node[below] (M) {\bf multinomial model};
        \filldraw (6,2) circle (0pt) node[above] (P) {\bf Poisson model};
        \draw[trans] (H) -- (B) node[midway,left] {$ \Binom(k,p) $ samples};
        \draw[trans] (H) -- (M) node[midway,below] {simulate};
        \draw[trans] (M) -- (P) node[midway,right] {$ \Poi(n) $ samples};
    \end{tikzpicture}
    \caption{Relations between the four sampling models.
                In particular, hypergeometric (resp.~multinomial) model reduces to the 
                Bernoulli (resp.~Poisson) model when the sample size is binomial (resp.~Poisson) distributed.
                \label{fig:model-connection}}
\end{figure}
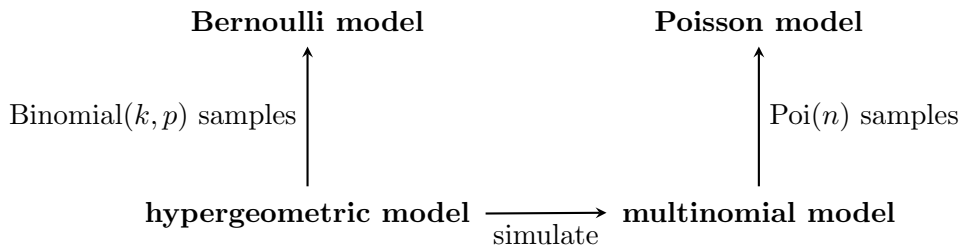

The connections between different models are illustrated in \prettyref{fig:model-connection}.
Under the Poisson model, the sample size is a Poisson random variable; conditioned on the sample size, the samples are \iid which is identical to the multinomial model. 
The same relation holds as the Bernoulli model to the hypergeometric model. 
Given samples $ (Y_1,\dots,Y_n) $ uniformly drawn from a $ k $-ball urn without replacement (hypergeometric model), we can simulate $ (X_1,\dots,X_n) $ drawn with replacement (multinomial model) as follows: for each $ i=1,\dots,n $, let
\begin{equation}
    X_i=\begin{cases}
        Y_i, & \text{with probability}~1-\frac{i-1}{k},\\
        Y_m, & \text{with probability}~\frac{i-1}{k}, \quad m \sim \text{Uniform}([i-1]).
    \end{cases}
    \label{eq:hyper-multi}
\end{equation}

In view of the connections in \prettyref{fig:model-connection}, any estimator constructed for one specific model can be adapted to another.
The adaptation from multinomial to hypergeometric model is provided by the simulation in \prettyref{eq:hyper-multi}, and the other direction is given by \prettyref{lmm:w-wo-replacement} (without modifying the estimator).
The following result provides a recipe for going between fixed and randomized sample size:
\begin{lemma}
    \label{lmm:randomized}
    Let $ N $ be an $\naturals$-valued  random variable.
    \begin{enumerate}[(a)]
        \item Given any\footnote{More precisely, here and below $\hat C$ is understood as a sequence of estimators indexed by the sample size $(X_1,\ldots,X_n)\mapsto \hat C(X_1,\ldots,X_n)$.} $ \hat{C} $ that uses $ n $ samples and succeeds with probability at least $ 1-\delta $, there exists $ \hat{C}' $ using $ N $ samples that succeeds with probability at least $ 1-\delta-\Prob[N<n] $.
        \item Given any $ \tilde{C} $ using $ N $ samples that succeeds with probability at least $ 1-\delta $, there exists $ \tilde{C}' $ that uses $ n $ samples and  succeeds with probability at least $ 1-\delta-\Prob[N>n] $.
    \end{enumerate}
\end{lemma}
\begin{proof}
    \begin{enumerate}[(a)]
        \item 
        Denote the samples by $ X_1,\dots,X_N $.
        Following \cite[Lemma 5.3(a)]{RRSS09},  define $ \hat{C}' $ as
        \begin{equation*}
            \hat{C}'=
            \begin{cases}
                \hat{C}(X_1,\dots,X_n), & N\ge n,\\
                0, & N<n.
            \end{cases}
        \end{equation*}
        Then $ \hat{C}' $ succeeds as long as $ N\ge n $ and $ \hat{C} $ succeeds, which has probability at least $ 1-\delta-\Prob[N<n] $.

        \item Denote the samples by $ X_1,\dots,X_{n} $.
        Draw a random variable $ m $ from the distribution of $ N $ and define $ \tilde{C}' $ as 
        \begin{equation*}
            \tilde{C}'=
            \begin{cases}
                \tilde{C}(X_1,\dots,X_m), & m\le n,\\
                0, & m>n.
            \end{cases}
        \end{equation*}
        The given estimator $ \tilde{C} $ fails with probability $ \sum_{j\ge 0}\Prob[\tilde{C}~\text{fails}|N=j]\Prob[N=j]\le \delta $.
        Consequently, $ \sum_{j=0}^{n}\Prob[\tilde{C}~\text{fails}|N=j]\Prob[N=j]\le \delta $.
        The estimator $ \tilde{C}' $ fails with probability at most
        \begin{equation*}
            \sum_{j=0}^{n}\Prob[\tilde{C}~\text{fails}|m=j]\Prob[m=j]+\Prob[m>n]
            \le \delta+\Prob[m>n],
        \end{equation*}
        which completes the proof.
    \end{enumerate}
\end{proof}

The adaptations of estimators between different sampling models imply the relations of the fundamental limits on the corresponding sample complexities.
Extending \prettyref{def:sample}, let $ n^*_M(k,\Delta,\delta) $, $ n^*_H(k,\Delta,\delta) $, $ n^*_B(k,\Delta,\delta) $, and $ n^*_P(k,\Delta,\delta) $ be the minimum expected sample size under the multinomial, hypergeometric, Bernoulli, and Poisson sampling model, respectively, such that there exists an estimator $ \hat{C} $ satisfying $ \Prob[|\hat{C}-C|\ge \Delta]\le \delta $.
Combining Chernoff bounds (see, \eg, \cite[Theorem 4.4, 4.5, and 5.4]{MU06}), we obtain \prettyref{cor:nstar-connection}, in which the connection between multinomial and Poisson models gives a rigorous justification of the assumption on the Poisson sampling model in \prettyref{sec:ub}.
\begin{corollary}
    \label{cor:nstar-connection} The following relations hold:
    \begin{itemize}
        \item $ n^*_H $ versus $ n^*_M $:
        \begin{itemize}
            \item[(a)] $ n^*_H(k,\Delta,\delta)\le n^*_M(k,\Delta,\delta) $;
            \item[(b)] $ n^*_H(k,\Delta,\delta)\le n \Rightarrow n^*_M(k',2\Delta,\delta+\binom{k'}{\Delta}(1-\frac{\Delta}{k'})^k)\le n $, for any $k'\in\naturals$.
                                                In particular, if $\delta$ is a constant, then we can choose $k'=\Theta(k/\log\frac{k}{\Delta})$.
        \end{itemize}

        \item $ n^*_P $ versus $ n^*_M $:
        \begin{itemize}
            \item[(c)] $ n_P^*(k,\Delta,\delta)\le n \Rightarrow n_M^*(k,\Delta,\delta+(e/4)^n)\le 2n $;
            \item[(d)] $ n_M^*(k,\Delta,\delta)\le n \Rightarrow n_P^*(k,\Delta,\delta+(2/e)^n)\le 2n $.
        \end{itemize}

        \item $ n^*_B $ versus $ n^*_H $:
        \begin{itemize}
            \item[(e)] $ n_B^*(k,\Delta,\delta)\le n \Rightarrow n_H^*(k,\Delta,\delta+(e/4)^n)\le 2n $;
            \item[(f)] $ n_H^*(k,\Delta,\delta)\le n \Rightarrow n_B^*(k,\Delta,\delta+(2/e)^n)\le 2n $.
        \end{itemize}
    \end{itemize}
\end{corollary}


\section{Correlation decay between fingerprints}
\label{app:corr}
Recall that the fingerprints are defined by $    \Phi_j=\sum_i \indc{N_i=j}$,
where $ N_i $ denotes the histogram of samples.
Under the Poisson model, $ N_i\inddistr\Poi(np_i) $. Then
\begin{align*}
   \cov(\Phi_j,\Phi_{j'})&=-\sum_{i}\Prob[N_i=j]\Prob[N_i=j'],\quad j\ne j',\\
   \var[\Phi_j]&=\sum_{i}\Prob[N_i=j](1-\Prob[N_i=j]).
\end{align*}
The correlation coefficient between $ \Phi_0 $ and $ \Phi_j $ follows immediately:
\begin{align}
  |\rho(\Phi_0,\Phi_j)|
  & =\sum_i\frac{\Prob[N_i=0]\Prob[N_i=j]}{\sqrt{\sum_{l}\Prob[N_l=0](1-\Prob[N_l=0])\sum_{l}\Prob[N_l=j](1-\Prob[N_l=j])}}\nonumber\\
  & \le \sum_i\frac{\Prob[N_i=0]\Prob[N_i=j]}{\sqrt{\Prob[N_i=0](1-\Prob[N_i=0])\Prob[N_i=j](1-\Prob[N_i=j])}}\nonumber\\
  & = \sum_i\sqrt{\frac{\Prob[N_i=0]}{1-\Prob[N_i=0]} \frac{\Prob[N_i=j]}{1-\Prob[N_i=j]}}
  = \sum_i\sqrt{\frac{e^{-\lambda_i}}{1-e^{-\lambda_i}}\frac{\frac{e^{-\lambda_i}\lambda_i^j}{j!}}{1-\frac{e^{-\lambda_i}\lambda_i^j}{j!}}},\label{eq:correlation}
\end{align}
where $ \lambda_i=np_i $.
Note that $ \max_{x>0}\frac{e^{-x}x^j}{j!}=\frac{e^{-j}j^j}{j!}\rightarrow 0 $ as $ j\diverge $.
Therefore, for any $ x>0 $, 
\begin{equation}
    \frac{e^{-x}}{1-e^{-x}}\frac{\frac{e^{-x}x^j}{j!}}{1-\frac{e^{-x}x^j}{j!}}
    = \frac{1}{j!}\frac{e^{-2x}x^j}{1-e^{-x}}(1+o_j(1)),
    \label{eq:fx}
\end{equation}
where $ o_j(1) $ is uniform as $ j\diverge $.
Taking derivative, the function $ x\mapsto \frac{e^{-2x}x^j}{1-e^{-x}} $ on $ x>0 $ is increasing if and only if $ x+e^x(j-2x)-j>0 $, and the maximum is attained at $ x=j/2+o_j(1) $.
Therefore, applying $ j!>(j/e)^j $,
\begin{equation}
    \frac{1}{j!}\frac{e^{-2x}x^j}{1-e^{-x}}
    \le (1+o_j(1))2^{-j}.
    \label{eq:fx-ub}
\end{equation}
Combining \prettyref{eq:correlation} -- \prettyref{eq:fx-ub}, we conclude that
\begin{equation*}
    |\rho(\Phi_0,\Phi_j)|
    \le k2^{-j/2}(1+o_j(1)).
\end{equation*}

\section{Proof of auxiliary lemmas}
\label{app:aux}
\begin{proof}[Proof of \prettyref{lmm:tm-ub}]
    For any $ z\in\complex $, we can represent the forward difference in \prettyref{eq:tn-full} as an integral:
    \begin{align*}
      \Delta^mf(z)
      &= f(z+m)-\binom{m}{1}f(z+m-1)+\dots+(-1)^mf(z) \\
      &= \int_{[0,1]^m}f^{(m)}(z+x_1+\dots+x_m)\diff x_1 \cdots \diff x_m.
    \end{align*}
    Therefore,
    \begin{equation}
        |t_m(z)|
        = \abs{\frac{1}{m!}\Delta^mp_m(z)}
        \le \frac{1}{m!}\sup_{0\le \xi \le m}|p_m^{(m)}(z+\xi)|.
        \label{eq:tm-ub-ref1}
    \end{equation}
    Recall the definition of $ p_m $ in \prettyref{eq:def-pm}. Let $ p_m(z)=\sum_{l=0}^{2m}a_\ell z^\ell $.
    Let $ z(z-1)\cdots(z-m+1)=\sum_{i=0}^{m}b_iz^i $ and $ (z-M)(z-M-1)\cdots(z-M-m+1)=\sum_{i=0}^{m}c_iz^i $.
    Expanding the product and collecting the coefficients yields a simple upper bound:
    \begin{equation*}
        |b_i|\le 2^m(m-1)^{m-i},\qquad
        |c_i|\le 2^m(M+m-1)^{m-i}\le 2^{m}(2M)^{m-i}\le 2^{2m}M^{m-i}.
    \end{equation*}
    Since $ \sum_{\ell=0}^{2m}a_\ell z^\ell = (\sum_{i=0}^{m}b_iz^i)(\sum_{j=0}^{m}c_jz^j)$ , for $ \ell\ge m $,
    \begin{align*}
      |a_\ell |
      &=\abs{\sum_{i=\ell-m}^{m}b_ic_{\ell-i}}
      \le \sum_{i=\ell-m}^{m}{2^{3m}(m-1)^{m-i}M^{m-\ell+i}}\\
      &= 2^{3m}M^{2m-\ell}\sum_{i=\ell-m}^{m}\pth{\frac{m-1}{M}}^{m-i}
      \le m2^{3m}M^{2m-\ell}.
    \end{align*}
    Taking $ m $-th derivative of $ p_m $, we obtain
    \begin{align*}
      |p_m^{(m)}(z)|
      &=\abs{\sum_{j=0}^{m}a_{j+m}\frac{(j+m)!}{j!}z^j}\\
      &\le \sum_{j=0}^{m}|a_{j+m}M^j|\binom{m+j}{m}m!\abs{\frac{z}{M}}^j
        \le m2^{3m}M^{m}m!(2e)^m\sum_{j=0}^{m}\abs{\frac{z}{M}}^j\\
      & \le m^22^{6m}M^{m}m!\pth{\frac{|z|}{M}\vee 1}^m
        =m^22^{6m}m!\pth{|z|\vee M}^m.
    \end{align*}
    Then the desired \prettyref{eq:tm-ub} follows from \prettyref{eq:tm-ub-ref1}.
\end{proof}

\begin{proof}[Proof of \prettyref{lmm:stirling}]
    The following uniform asymptotic expansions of the Stirling numbers of the first kind was obtained in \cite[Theorem 2]{CRT00}:
    \begin{equation*}
        |s(n+1,m+1)|=
        \begin{cases}
            \frac{n!}{m!}(\log n+\gamma)^m(1+o(1)),& 1\le m\le \sqrt{\log n},\\
            \frac{\Gamma(n+1+R)}{\Gamma(R)R^{m+1}\sqrt{2\pi H}}(1+o(1)),& \sqrt{\log n}\le m\le n-n^{1/3},\\
            \binom{n+1}{m+1}(\frac{m+1}{2})^{n-m}(1+o(1)), & n-n^{1/3}\le m \le n,
        \end{cases}
    \end{equation*}
    where $ \gamma $ is Euler's constant, $ R $ is the unique positive solution to $ h'(x)=0 $ with $ h(x)\triangleq \log\frac{\Gamma(x+n+1)}{\Gamma(x+1)x^{m}} $, $ H=R^2h''(R) $, and all $ o(1) $ terms are uniform in $ m $.
    In the following we consider each range separately and prove the non-asymptotic approximation in \prettyref{eq:stirling}.

    Case I. For $ 1\le m\le \sqrt{\log n} $, Stirling's approximation gives
    \begin{equation*}
        \frac{n!}{m!}(\log n+\gamma)^m
        =n!\pth{\Theta\pth{\frac{\log n}{m}}}^m.
    \end{equation*}

    Case II. For $ n-n^{1/3}\le m \le n $,
    \begin{align*}
      \binom{n+1}{m+1}\pth{\frac{m+1}{2}}^{n-m}
      & = \frac{n!}{m!}\pth{\Theta\pth{\frac{m}{n-m}}}^{n-m}\\
      & = n!\exp\pth{m\pth{\frac{n-m}{m}\log\pth{\Theta\pth{\frac{m}{n-m}}}-\log \Theta(m)}}\\
      & = n!\pth{\Theta\pth{\frac{1}{m}}}^m.
    \end{align*}

    Case III. For $ \sqrt{\log n}\le m\le n-n^{1/3} $,
    note that $ h(x)=\sum_{i=1}^n\log(x+i)-m\log x $, and thus $ H=R^2h''(R)=m-\sum_{i=1}^{n}\frac{R^2}{(R+i)^2}\le m $.
    By \cite[Lemma 4.1]{MW58}, $ H=\omega(1) $ in this range.
    Hence,
    \begin{equation}
        |s(n+1,m+1)|=\frac{\Gamma(n+1+R)}{\Gamma(R)R^{m+1}}(\Theta(1))^m= \frac{n!}{R^m}\frac{\Gamma(n+1+R)}{n!\Gamma(R+1)}(\Theta(1))^m,
        \label{eq:stirling-ref2}
    \end{equation}
    where $ R $ is the solution to  $ x(\frac{1}{x+1}+\dots+\frac{1}{x+n})=m $.
    Bounding the sum by integrals, we have
    \begin{equation*}
        R\log\pth{1+\frac{n}{R+1}}\le m \le R\log\pth{1+\frac{n}{R}}.
    \end{equation*}
    If $ \sqrt{\log n}\le m\le \frac{n}{e} $, then $ R\asymp \frac{m}{\log(n/m)} $, and hence
    \begin{equation*}
        1\le \frac{\Gamma(n+1+R)}{n!\Gamma(R+1)}
        \le \pth{O\pth{\frac{n+R}{R}}}^R
        = \exp(O(m)).
    \end{equation*}
    In view of \prettyref{eq:stirling-ref2}, we have $ |s(n+1,m+1)|=\frac{n!}{(\Theta(R))^m} $, which is exactly \prettyref{eq:stirling} when $ m\le n/e $.
    If $ n/e \le m \le n-n^{1/3} $, then $ R\asymp \frac{n^2}{n-m} $, and
    \begin{align*}
      \frac{1}{R^m}\frac{\Gamma(n+1+R)}{n!\Gamma(R+1)}
      & =R^{-m} \pth{\Theta\pth{\frac{n+R}{n}}}^n\\
      & =\exp\pth{-m\log\Theta\pth{\frac{n^2}{n-m}}+n\log\Theta\pth{\frac{n}{n-m}}} \\
      & =\exp\pth{-m\log \Theta(n)+(n-m)\log\Theta\pth{\frac{n}{n-m}}}\\
      & =\exp\pth{-m\log \Theta(n)}.
    \end{align*}
    Combining \prettyref{eq:stirling-ref2} yields that $ |s(n+1,m+1)|=n!(\Theta(\frac{1}{n}))^m $, which coincides with \prettyref{eq:stirling} since $ n\asymp m $ is this range.
\end{proof}



%% file: species.bbl
\newcommand{\etalchar}[1]{$^{#1}$}